\documentclass[10pt]{article}

\DeclareMathAlphabet{\mathpzc}{OT1}{pzc}{m}{it}

\usepackage{amsmath}
\usepackage{amssymb}
\usepackage{amsfonts}
\usepackage{amsthm}
\usepackage{mathrsfs}
\usepackage{ifsym}
\usepackage{fontenc}
\usepackage{textcomp}
\usepackage{tipa}
\usepackage{enumerate}
\usepackage[pdftex]{color, graphicx}
\numberwithin{equation}{section}
\usepackage{hyperref}

\begin{document}
 
\title{{\bf The Hasse invariant of the Tate normal form $E_7$ and the supersingular polynomial for the Fricke group $\Gamma_0^*(7)$}}         
\author{Patrick Morton}        
\date{June 3, 2023}          
\maketitle

\begin{abstract}  A formula is proved for the number of linear factors and irreducible cubic factors over $\mathbb{F}_l$ of the Hasse invariant $\hat H_{7,l}(a)$ of the elliptic curve $E_7(a)$ in Tate normal form, on which the point $(0,0)$ has order $7$, as a polynomial in the parameter $a$, in terms of the class number of the imaginary quadratic field $K=\mathbb{Q}(\sqrt{-l})$.  Conjectural formulas are stated for the numbers of quadratic and sextic factors of $\hat H_{7,l}(a)$ of certain specific forms in terms of the class number of $\mathbb{Q}(\sqrt{-7l})$, which are shown to imply a recent conjecture of Nakaya on the number of linear factors over $\mathbb{F}_l$ of the supersingular polynomial $ss_l^{(7*)}(X)$ corresponding to the Fricke group $\Gamma_0^*(7)$.
\end{abstract}

\section{Introduction}

This paper is a continuation of the discussion begun in \cite{mor1} and \cite{mor2}.  It concerns a conjecture of Nakaya on the supersingular polynomial $ss_p^{(N*)}(X)$ for the Fricke group $\Gamma_0^*(N)$ and its relationship to the Hasse invariant of the Tate normal form $E_7$ for a point of order $7$.  Recall that the Fricke group $\Gamma_0^*(N)$ is
$$\Gamma_0^*(N) = \Gamma_0(N) \cup \Gamma_0(N) W_N, \ \ W_N = \left(\begin{array}{cc}0 & \frac{-1}{\sqrt{N}} \\ \sqrt{N} & 0\end{array}\right),$$
where $W_N \in SL_2(\mathbb{R})$ is the matrix of the involution $\tau \rightarrow \frac{-1}{N\tau}$.  The supersingular polynomial $ss_p^{(N*)}(X)$ was defined originally for $N \in \{2, 3, 5, 7\}$ by Koike and Sakai \cite{sa1}, \cite{sa2} and for $N$ in the set of prime divisors $\mathfrak{S}$ of the monster group by Nakaya \cite{na}, as follows.  If $N$ is an element of the set
$$\mathfrak{S} = \{2, 3, 5, 7, 11, 13, 17, 19, 23, 29, 31, 41, 47, 59, 71\},$$
then there is a Hauptmodul $j_N^*(\tau)$ for the modular functions invariant under the Fricke group $\Gamma_0^*(N)$, and the modular functions $j(\tau), j(N\tau)$ satisfy a quadratic equation over the field $\mathbb{Q}(j_N^*(\tau))$, given by
\begin{align*}
& \ R_N(X,Y) := X^2 - a_N(Y) X +b_N(Y) \in \mathbb{Z}[X,Y];\\
& \ R_N(j(\tau),j_N^*(\tau)) = R_N(j(N\tau),j_N^*(\tau)) = 0.
\end{align*}
In the case $N = 7$ under consideration in this paper, we have
\begin{align*}
R_7(X,Y) &= X^2-XY(Y^2-21Y+8)(Y^4-42Y^3+454Y^2-1008Y-1280)\\
& \ \ +Y^2(Y^2+224Y+448)^3
\end{align*}
and
\begin{align*}
j_7^*(\tau) = & \ \left(\frac{\eta(\tau)}{\eta(7\tau)}\right)^4+13+49\left(\frac{\eta(7\tau)}{\eta(\tau)}\right)^4\\
= & \ q^{-1}+9 + 51q+204q^2+681q^3+1956q^4+5135q^5+ \cdots, \ \ q = e^{2\pi i \tau}.
\end{align*}
This function shows up in Ramanujan's notebooks.  See \cite[Entry 32(iii), (32.8), pp. 176-177]{ber} and \cite[(4.14)]{bz}.  The supersingular invariants $j_N^*$ for $\Gamma_0^*(N)$ are the values in $\overline{\mathbb{F}}_p$ for which there is a supersingular $j$-invariant $j \in \mathbb{F}_{p^2}$ satisfying
$$R_N(j,j_N^*) = j^2 - a_N(j_N^*) j +b_N(j_N^*) = 0 \ \textrm{in} \ \overline{\mathbb{F}}_p.$$
If $ss_p(X)$ is the ordinary supersingular polynomial, then the supersingular polynomial $ss_p^{(N*)}(X)$ for $\Gamma_0^*(N)$ is the product
$$ss_p^{(N*)}(X) = \prod_{R_N(j,j_N^*) = ss_p(j) = 0}{(X-j_N^*)} \ \ \textrm{in} \ \mathbb{F}_p[X]$$
over the {\it distinct} values $j_N^*$, as $j$ ranges over all supersingular $j$-invariants in $\mathbb{F}_{p^2}$.  This definition is similar to the definition of the polynomial $ss_p^{(N)}(X)$ for $\Gamma_0(N)$ and $N \in \{2,3,4\}$ given by Tsutsumi \cite{tsut}.  \medskip

A formula for $ss_p^{(N*)}(X)$ in terms of resultants, conjectured by Nakaya, is proved in \cite{mor3}.  In the case $N=7$ this formula is:
\begin{align*}
&(Y+1)^{\mu_7}(Y-27)^{\mu_7} \textrm{Res}_X(ss_p(X),R_7(X,Y)) \equiv \\
& (Y^2+224Y+448)^{2\delta}(Y^4-528Y^3-9024Y^2-5120Y-1728)^\varepsilon ss_p^{(7*)}(Y)^2
\end{align*}
modulo $p$, where
$$ \mu_7 = \frac{1}{2}\left(1-\left(\frac{-7}{p}\right)\right), \ \ \delta = \frac{1}{2}\left(1-\left(\frac{-3}{p}\right)\right), \ \ \varepsilon = \frac{1}{2}\left(1-\left(\frac{-4}{p}\right)\right).$$
If $N \in \{2, 3, 5, 7\}$, then $j_N^* \in \mathbb{F}_{p^2}$, by \cite[Thm. 6.1]{mor1}.  For examples of these polynomials see \cite[Tables 2,3]{mor3}.  Sakai \cite[Conj. 5.1]{sa2} has conjectured that $ss_p^{(7*)}(X)$ is congruent (mod $p$) to one of a sequence of orthogonal polynomials $\{A_n^{(N)}(X)\}_{n \ge 0}$, which he constructed using an inner product on the set of holomorphic modular functions (on the upper half-plane) for $\Gamma_0^*(7)$.  These polynomials, which he calls the Atkin orthogonal polynomials for $\Gamma_0^*(7)$, are obtained using the Gram-Schmidt process with respect to this inner product applied to the sequence $\{j_7^*(\tau)^n\}$.  Sakai \cite[Conj. 5.2, p. 2255]{sa2} and Nakaya \cite[Conj. 1, p. 489]{na} also conjecture that $ss_p^{(7*)}(X)$ is congruent to a specific solution of Heun's differential equation.
\medskip

In this paper we are concerned with the following conjecture of Nakaya.  \medskip

\noindent {\bf Nakaya's conjecture on linear factors.} (\cite[Conjecture 5]{na})
{\it Let $ p \ge 5$ be a prime number and
$$N \in \mathfrak{S} = \{2, 3, 5, 7, 11, 13, 17, 19, 23, 29, 31, 41, 47, 59, 71\}$$
be a prime divisor of the order of the Monster group, with $N \neq 2, p$.  Then the number of linear factors $L^{(N*)}(p)$ of the supersingular polynomial $ss_p^{(N*)}(X)$ for the Fricke group $\Gamma_0^*(N)$ is given by
\begin{align*}
L^{(N*)}(p) = & \ \frac{1}{2}\left(1+\left(\frac{-p}{N}\right) \right) L(p)\\
& \ +\frac{1}{8}\Big\{2+\left(1-\left(\frac{-1}{Np}\right) \right) \left(2+\left(\frac{-2}{Np}\right) \right) \Big\} h(-Np),
\end{align*}
where $L(p)$ is the number of supersingular $j$-invariants of elliptic curves in characteristic $p$ which lie in the prime field $\mathbb{F}_p$ and $h(-Np)$ is the class number of the field $\mathbb{Q}(\sqrt{-Np})$.}
\bigskip

Nakaya proved similar formulas for $N=2, 3$ in \cite[Thm. 4]{na} and used this to show that the primes $p$ for which $ss_p^{(2*)}(X)$ or $ss_p^{(3*)}(X)$ are products of linear factors (mod $p$) are the primes dividing the orders of certain sporadic simple groups, namely, the baby monster $B$ and the Fischer group $Fi_{24}$, respectively, in the notation of \cite[pp. 296-297]{con}.  (A simpler proof of the connection with the orders of these simple groups is given in \cite{mor3}.) \medskip

A proof of Nakaya's conjecture for $N=5$ is given in \cite{mor1} and \cite{mor2} which depends on knowing the number of irreducible factors of various types of the Hasse invariant $\hat H_{5,p}(X)$ in characteristic $p$ for the elliptic curve $E_5$, the Tate normal form for a point of order $5$.  The results necessary for establishing the conjecture in this case were proved in \cite{mor1} (for linear factors and quadratic factors with constant term $-1$) and \cite{mor2} (for certain quartic and related quadratic factors). \medskip

In this paper I will show that Nakaya's conjecture for $N=7$ also depends on knowing the numbers of irreducible factors of certain types of the Hasse invariant $\hat H_{7,p}(X)$ for the Tate normal form $E_7$ for a point of order $7$.  Specifically, I show that Nakaya's conjecture follows from formulas for the number of linear factors, the number of irreducible cubic factors, and the numbers of irreducible quadratic and sextic factors of specific forms dividing $\hat H_{7,p}(X)$ in characteristic $p$.  Conjectural formulas for these numbers in terms of the class numbers $h(-p)$ and $h(-7p)$ of the quadratic fields $K = \mathbb{Q}(\sqrt{-p})$ and $K' = \mathbb{Q}(\sqrt{-7p})$ are stated in Section 2, and the proof that Nakaya's conjecture for $N=7$ follows from Conjectures \ref{conj:1}-\ref{conj:3} is given in Section 3. \medskip

In the remainder of the paper I give a proof of Conjecture \ref{conj:1} on the numbers of linear and irreducible cubic factors of $\hat H_{7,p}(X)$.  The results for these factors can be stated as follows.  For notational convenience and to align with the notation of \cite{mor1} I prove this conjecture using the letter $l$ in place of $p$ for the characteristic under consideration. See \cite[Conj. 1, p. 260]{mor1}. \bigskip

\noindent {\bf Theorem A.} {\it Let $l \neq 2, 3, 7$ be a prime and denote by $h(-l)$ the class number of the quadratic field $K = \mathbb{Q}(\sqrt{-l})$.
\begin{enumerate}[a)]
\item If $l \equiv 6$ (mod $7$), the number $N_1$ of distinct linear factors in the parameter $a$ which divide the Hasse invariant $\hat H_{7,l}(a)$ of the Tate normal form $E_7(a)$ in characteristic $l$ is
\begin{align*}
N_1 = \begin{cases}
3h(-l), \ &\textrm{if} \ \ l \equiv 1 \ (\textrm{mod} \ 4);\\
3\left(3-\left(\frac{2}{l}\right)\right)h(-l), \ &\textrm{if} \ \ l \equiv 3 \ (\textrm{mod} \ 4).
\end{cases}
\end{align*}
This is six times the number of supersingular $j$-invariants which lie in $\mathbb{F}_l$.
\item If $l \equiv 3,5$ (mod $7$), the number $N_3$ of distinct irreducible cubics which divide $\hat H_{7,l}(a)$ in characteristic $l$ is
\begin{align*}
N_3 = \begin{cases}
h(-l), \ &\textrm{if} \ \ l \equiv 1 \ (\textrm{mod} \ 4);\\
\left(3-\left(\frac{2}{l}\right)\right)h(-l), \ &\textrm{if} \ \ l \equiv 3 \ (\textrm{mod} \ 4).
\end{cases}
\end{align*}
This is twice the number of supersingular $j$-invariants which lie in $\mathbb{F}_l$.
\end{enumerate}}
\label{thm:A}
\medskip

Part (a) of this theorem is the analogue for $N=7$ of \cite[Thm. 1.1]{mor1} for the case $N=5$ and Part (b) is the analogue of \cite[Thm. 1.3]{mor1}.  Note that the primes mentioned in this theorem are the only primes for which $\hat H_{7,l}(X)$ has linear or irreducible cubic factors, by \cite[Thm. 5.2]{mor1}. \medskip

The key to the proof of the theorem is showing that certain class fields over the field $K = \mathbb{Q}(\sqrt{-l})$, for primes $l$ satisfying $\left(\frac{-l}{7}\right) = +1$, are generated by values of the modular function
$$h(\tau) = q^{-1} \prod_{n \ge 1}{\frac{(1-q^{7n-3})(1-q^{7n-4})(1-q^{7n-2})^2(1-q^{7n-5})^2}{(1-q^{7n-1})^3(1-q^{7n-6})^3}}, \ q = e^{2 \pi i \tau}.$$
See \cite{du}.  Specifically, let $(7)= \wp_7 \wp_7'$ be the factorization of the ideal $(7)$ in the ring of integers $R_K$ of $K$.  The class fields in question are $\Sigma_{\wp_7'} \Omega_f$ and $\Sigma_{\wp_7} \Omega_f$, where $\Sigma_\mathfrak{f}$ is the ray class field of conductor $\mathfrak{f}$ over $K$ and $\Omega_f$ is the ring class field of conductor $f$, where $f = 1$ or $2$ if $l \equiv 3$ (mod $4$) and $f = 1$ if $l \equiv 1$ (mod $4$).  Putting $-d = d_K f^2$, where $d_K$ is the discriminant of $K$, let
$$w = \frac{v+\sqrt{-d}}{2},  \ \ \wp_7^2 \mid w, \ \ (w,f) = 1.$$
Then the results of Theorems \ref{thm:5} and \ref{thm:6} in Section 5 are that
\begin{equation}
\Sigma_{\wp_7'} \Omega_f = \mathbb{Q}(h(w/7)) \ \ \textrm{and} \ \ \Sigma_{\wp_7} \Omega_f = \mathbb{Q}(h(-1/w)).
\label{eqn:S}
\end{equation}
This is analogous to the result of \cite[Thm. 1.1]{mor4}.  These facts are used in Theorem \ref{thm:n1} to show how class field theory provides the link between the modular function relation $R_7(j(\tau),j_7^*(\tau)) = 0$ and the definition of the supersingular invariants for $\Gamma_0^*(7)$. \medskip

In order to prove these results, some important relationships between modular functions are collected in Section 4.  In particular, I show that the function $h(\tau)$ satisfies $f_7(h(\tau),j_7^*(\tau)) = 0$, where
\begin{align*}
f_7(x,t)&=x^6-(3+t) x^5+(6+9 t) x^4-(7+13 t) x^3+(6+4 t) x^2\\
& \ \ \ +(-3+t) x +1\\
& = (x^2 - x + 1)^3 - t x(x - 1)(x^3 - 8x^2 + 5x + 1)
\end{align*}
gives the form of the irreducible $6$-th degree factors of $\hat H_{7,l}(X)$ which are counted in Conjecture \ref{conj:2}. \medskip

In Section 6 the relationship between various fields of modular functions are used to show that the polynomial in two variables $z$ and $j$,
$$F(z,j) = (z^2-3z+9)(z^2-11z+25)^3-j(z-8),$$
has Galois group equal to $\textsf{PGL}(2,7)$ over $\mathbb{Q}(j)$ (see Lemma \ref{lem:3} in Section 6), from which it follows that $F(z,j) \equiv 0$ (mod $l$) has at most two solutions $z \in \mathbb{F}_l$, for any supersingular $j$-invariant $j \neq 0, 1728$ in $\mathbb{F}_l$ and $l \equiv 3, 5, 6$ (mod $7$).  Together with the facts in (\ref{eqn:S}), this is used to prove Theorem \ref{thm:7}, that the polynomial
\begin{align*}
G(x,j) = & (x^2 - x + 1)^3 (x^6 + 229x^5 + 270x^4 - 1695x^3 + 1430x^2 - 235x + 1)^3\\
&\ - jx(x-1)(x^3 - 8x^2 + 5x + 1)^7
\end{align*}
has exactly six linear factors (mod $l$) if $l \equiv 6$ (mod $7$) and two irreducible cubic factors (mod $l$) if $l \equiv 3, 5$ (mod $7$), for any supersingular $j$-invariant $j \neq 0, 1728$ in characteristic $l$.  Then Theorem A follows from a formula for the Hasse invariant (see Section 2). \medskip

In Section 7 I prove as a corollary to Theorem \ref{thm:7} that the powers of $h(w/7)$, with $w$ as above, form an $\mathfrak{l}$-integral basis of $\Sigma_{\wp_7'} \Omega_f/K$, where $\mathfrak{l}^2 = (l)$ in $K$.  This follows from the fact that the minimal polynomial $P_d(x)$ of $h(w/7)$ over $\mathbb{Q}$ factors (mod $l$) as a product of the squares of distinct linear or distinct irreducible cubic factors in cases (a) and (b) of Theorem A, respectively.  In the Appendix I give a proof of an important known identity in Theorem \ref{thm:10} relating $h(\tau)$ and the Dedekind $\eta$-function $\eta(\tau)$.  The proof is included here for the convenience of the reader. \medskip

This paper is dedicated to John Brillhart, my undergraduate advisor and co-author, who passed away in May, 2022.  His inspiring teaching and playful attitude towards research and all things mathematical helped to put my own mathematical career and joy for doing mathematics on a firm footing.  Long before undergraduate research projects were popular, the two papers we wrote together while I was an undergraduate (one with John Lomont) taught me the art of successful research and good mathematical writing, and gave me a powerful head start in graduate school.  By 1971, before I attended the University of Arizona as an undergraduate, John had empirically discovered the relations that we proved in the paper \cite[Thm. 1(a,c,d)]{brm}.  We kept these relations under wraps until we were able to learn enough to find a proof, more than $30$ years after he discovered them.  I am sure he would be especially pleased to see the congruences in Theorem \ref{thm:9} and the corresponding congruences in \cite[Thm. 4.1, Cor. 4.4]{mor1}, since polynomial congruences and the beautiful material in van der Waerden's ``Modern Algebra'' were often a subject of our conversations. 

\section{Conjectures for $E_7$}

The Tate normal form for a point of order $7$ is the elliptic curve
$$E_7: \ \ Y^2+(1+a-a^2)XY+(a^2-a^3)Y=X^3+(a^2-a^3)X^2,$$
whose $j$-invariant is
$$j_7(a) = \frac{(a^2-a+1)^3(a^6-11a^5+30a^4-15a^3-10a^2+5a+1)^3}{a^7(a-1)^7(a^3-8a^2+5a+1)},$$
(see \cite{du}); and whose Hasse invariant is the value at $x=a$ of the polynomial
\begin{align*}
\hat H_{7,l}(x) & = (x^2-x+1)^r(x^6-11x^5+30x^4-15x^3-10x^2+5x+1)^r\\
& \ \times (x^{12}-18x^{11}+117x^{10}-354x^9+570x^8-486x^7+273x^6\\
& \ -222x^5+174x^4-46x^3-15x^2+6x+1)^s\\
& \ \times x^{7n_l}(x-1)^{7n_l}(x^3-8x^2+5x+1)^{n_l}J_l(j_7(x))
\end{align*}
in characteristic $l$.  This expression for $\hat H_{7,l}(x)$ follows from the formulas in \cite[p. 236]{mor0}.  Here, as before,
\begin{align*}
& r  = r_l =  \frac{1}{2} \left(1-\left(\frac{-3}{l} \right) \right), \ \ \ s = s_l =  \frac{1}{2} \left(1-\left(\frac{-4}{l} \right) \right), \ \ \ n_l = \lfloor p/12 \rfloor;\\
& J_l(t) \equiv \sum_{k=0}^{n_l}{{2n_l + s\atopwithdelims ( ) 2k + s}{2n_l-2k \atopwithdelims ( ) n_l-k}(-432)^{n_l-k}(t-1728)^k} \ \ (\textrm{mod} \ l).
\end{align*}

Recall the following facts from \cite[Thm. 5.2.]{mor1}. \bigskip

\newtheorem{thm}{Theorem}

\begin{thm} If $l \neq 2, 3, 7$ is prime, the irreducible factors of $\hat H_{7,l}(x)$ over $\mathbb{F}_l$ are:
\begin{enumerate}[i)]
\item quadratic, if $l \equiv 1$ (mod $7$);
\item linear or quadratic, if $l \equiv 6$ (mod $7$);
\item $x^2-x+1$ \ or \ sextic, if $l \equiv 2,4$ (mod $7$);
\item $x^2-x+1$, cubic or sextic, if $l \equiv 3,5$ (mod $7$).
\end{enumerate}
\label{thm:1}
\end{thm}

We first state the following conjectures.  We will prove Conjecture 1 in Section 6 of this paper and leave the discussion of Conjectures 2 and 3 to the sequel.

\newtheorem{conj}{Conjecture}

\begin{conj}{(\cite[p.260]{mor1})} Let $l \neq 2, 3, 7$ be a prime.
\begin{enumerate}[a)]
\item If $l \equiv 6$ (mod $7$), the number $N_1$ of distinct linear factors which divide $\hat H_{7,l}(x)$ is
\begin{align*}
3h(-l), \ \textrm{if} \ \ l \equiv 1 \ (\textrm{mod} \ 4);\\
3\left(3-\left(\frac{2}{l}\right)\right)h(-l), \ \textrm{if} \ \ l \equiv 3 \ (\textrm{mod} \ 4).
\end{align*}
This is six times the number of supersingular $j$-invariants which lie in $\mathbb{F}_l$.
\item If $l \equiv 3,5$ (mod $7$), the number $N_3$ of distinct irreducible cubics which divide $\hat H_{7,l}(x)$ is
\begin{align*}
h(-l), \ \textrm{if} \ \ l \equiv 1 \ (\textrm{mod} \ 4);\\
\left(3-\left(\frac{2}{l}\right)\right)h(-l), \ \textrm{if} \ \ l \equiv 3 \ (\textrm{mod} \ 4).
\end{align*}
This is twice the number of supersingular $j$-invariants which lie in $\mathbb{F}_l$.
\end{enumerate}
\label{conj:1}
\end{conj}

\begin{conj} For a prime $l>7$ with $l \equiv 2, 3, 4, 5$ (mod $7$), the number $N_6$ of irreducible factors of the form
\begin{align*}
f_7(x,t)&=x^6-(3+t) x^5+(6+9 t) x^4-(7+13 t) x^3+(6+4 t) x^2+(-3+t) x +1\\
& = (x^2 - x + 1)^3 - t x(x - 1)(x^3 - 8x^2 + 5x + 1), \ \ t \in \mathbb{F}_l,
\end{align*}
dividing $\hat H_{7,l}(x)$ is a simple function of the class number $h(-7l)$ which depends on the congruence class of $l$ modulo $24 \cdot 7=168$.  Specifically:
\begin{enumerate}[A)]
\item If $l \equiv 2, 4$ (mod $7$), $N_6$ is given by
\begin{align*}
& \frac{1}{2}h(-7l)-\frac{1}{2}\left(1-\left(\frac{-3}{l}\right) \right), \ \textrm{if} \ \ l \equiv 1 \ (mod \ 8);\\
& h(-7l)-\frac{1}{2}\left(1-\left(\frac{-3}{l}\right) \right), \ \textrm{if} \ \ l \equiv 5 \ (mod \ 8);\\
& \frac{1}{4}h(-7l)-\frac{1}{2}\left(1-\left(\frac{-3}{l}\right) \right), \ \textrm{if} \ \ l \equiv 3 \ (mod \ 4);
\end{align*}
\item If $l \equiv 3, 5$ (mod $7$), $N_6$ is given by
\begin{align*}
& \frac{1}{2}h(-7l)-\frac{1}{2}\left(3-\left(\frac{-3}{l}\right) \right), \ \textrm{if} \ \ l \equiv 1 \ (mod \ 8);\\
& h(-7l)-\frac{1}{2}\left(3-\left(\frac{-3}{l}\right) \right), \ \textrm{if} \ \ l \equiv 5 \ (mod \ 8);\\
& \frac{1}{4}h(-7l)-\frac{1}{2}\left(3-\left(\frac{-3}{l}\right) \right), \ \textrm{if} \ \ l \equiv 3 \ (mod \ 4).
\end{align*}
\end{enumerate}
\label{conj:2}
\end{conj}

\begin{conj} For a prime $l>7$ with $l \equiv 1,6$ (mod $7$), the number $N_2$ of irreducible quadratics $x^2+ax+b$ which divide $\hat H_{7,l}(x)$ modulo $l$ and satisfy $B(a,b) \equiv 0$ (mod $l$), where
$$B(x,y)=x^3 + (-5y + 8)x^2 + (-8y^2 + 6y + 5)x - y^3 - 5y^2 + 8y - 1,$$
is a function of $h(-7l)$.  This is equivalent to $a \equiv (\alpha - 1)b - \alpha$ (mod $l$), for some root $\alpha \in \mathbb{F}_l$ of $x^3 - 8x^2 + 5x + 1 \equiv 0$ (mod $l$).  Specifically:
\begin{enumerate}[A)]
\item If $l \equiv 1$ (mod $7$), $N_2$ is given by
\begin{align*}
& \frac{3}{2}h(-7l)-\left(1-\left(\frac{-3}{l}\right) \right), \ \textrm{if} \ \ l \equiv 1 \ (mod \ 8);\\
& 3h(-7l)-\left(1-\left(\frac{-3}{l}\right) \right), \ \textrm{if} \ \ l \equiv 5 \ (mod \ 8);\\
& \frac{3}{4}h(-7l)-\left(1-\left(\frac{-3}{l}\right) \right), \ \textrm{if} \ \ l \equiv 3 \ (mod \ 4).
\end{align*}
\item If $l \equiv 6$ (mod $7$), $N_2$ is given by
\begin{align*}
& \frac{3}{2}h(-7l)-\left(4-\left(\frac{-3}{l}\right) \right), \ \textrm{if} \ \ l \equiv 1 \ (mod \ 8);\\
& 3h(-7l)-\left(4-\left(\frac{-3}{l}\right) \right), \ \textrm{if} \ \ l \equiv 5 \ (mod \ 8);\\
& \frac{3}{4}h(-7l)-\left(4-\left(\frac{-3}{l}\right) \right), \ \textrm{if} \ \ l \equiv 3 \ (mod \ 4).
\end{align*}
\end{enumerate}
\label{conj:3}
\end{conj}

The discussion in Section 3 will show that the quadratic polynomials $x^2+ax+b$ in Conjecture \ref{conj:3} are factors of the polynomial $f_7(x,t)$ in Conjecture \ref{conj:2}, for certain values of $t = j_7^* \in \mathbb{F}_l$.  Hence, Conjectures \ref{conj:2} and \ref{conj:3} are analogous to Theorem 1.1 in \cite{mor2}.

\section{Counting linear factors of $ss_p^{(7*)}(X)$}

In this section we show that Conjectures \ref{conj:1}-\ref{conj:3} imply Nakaya's Conjecture 5 for $N=7$. \smallskip

We shall use the parametrization
\begin{align}
\notag X =& \frac{(d^2-d+1)^3(d^6-11d^5+30d^4-15d^3-10d^2+5d+1)^3}{d^7(d-1)^7(d^3-8d^2+5d+1)},\\
\label{eqn:0} Y = & \frac{(d^2-d+1)^3}{d(d-1)(d^3-8d^2+5d+1)}, \ \ \ (N=7),
\end{align}
of the curve
\begin{align*}
R_7(X,Y) &= X^2-XY(Y^2-21Y+8)(Y^4-42Y^3+454Y^2-1008Y-1280)\\
& \ \ +Y^2(Y^2+224Y+448)^3.
\end{align*}
The roots of $ss_p^{(7*)}(X)$ over $\mathbb{F}_p$ are the values $Y = j_7^*$ for which $X=j$ is a supersingular $j$-invariant.  These are in turn the values of $X$ for which $d$ is a root of the Hasse invariant $\hat H_{7,p}(d)$ of the curve $E_7$. \medskip

It is clear that $j_7^* \in \mathbb{F}_p$ whenever $d \in \mathbb{F}_p$.  On the other hand, $j_7^* \in \mathbb{F}_p$ if and only if the supersingular parameter $d$ is a root of the polynomial
\begin{align*}
f_7(x,t)=& \ (x^2-x+1)^3-t x(x-1)(x^3-8x^2+5x+1)\\
= & \ x^6+(-3-t) x^5+(6+9 t) x^4+(-7-13 t) x^3+(6+4 t) x^2\\
 & \ +(-3+t) x +1,
\end{align*}
where $t=j_7^*$.  The irreducible sextics of this form are counted in Conjecture \ref{conj:2}.  However, this polynomial can also factor into a product of linear, quadratic or cubic factors.  By \cite[Thm. 5.3]{mor1}, all irreducible cubic factors of $\hat H_{7,p}(x)$ have the form
$$g(x,a) = x^3+ax^2-(a+3)x+1.$$
By the calculation
$$\textrm{Res}_x(g(x,a),f_7(x,t))=((a+8)t+a^2 + 3a + 9)^3$$
it is clear that $j_7^*= t = -\frac{a^2+3a+9}{a+8} \in \mathbb{F}_p$ whenever $g(x,a) \in \mathbb{F}_p[x]$ is an irreducible cubic dividing $\hat H_{7,p}(x)$; and $g(x,a) \mid f_7(x,t)$ for this value of $t$.  (For this note that $a \neq -8$ (mod $p$), since the corresponding cubic $x^3-8x^2+5x+1$ is a factor of $\Delta_7(x)=x^7(x-1)^7(x^3-8x^2+5x+1)$.  Hence, this cubic does not appear in the factorization of $\hat H_{7,p}(x)$.)  Thus, all irreducible cubic factors of $\hat H_{7,p}(x)$ divide a unique polynomial $f_7(x,t)$, with $t \in \mathbb{F}_p$.\medskip

Now consider irreducible quadratic factors.

\newtheorem{lem}{Lemma}

\begin{lem}
An irreducible quadratic factor $x^2+ax+b$ of $\hat H_{7,p}(x)$ divides $f_7(x,t)$ for some value of $t=j_7^* \in \mathbb{F}_p$ if and only if $B(a,b) \equiv 0$ (mod $p$), where $B(x,y)$ is defined in Conjecture \ref{conj:3}.
\label{lem:0}
\end{lem}

\begin{proof}
We compute the resultant
\begin{align*}
&\textrm{Res}_x(x^2+ax+b,f_7(x,t))=\\
&-b(b + 1 + a)\big\{a^3 + (-5b + 8)a^2 + (-8b^2 - 43b + 5)a- b^3 - 54b^2 - 41b - 1\big\}t^2\\
&+\big\{(-b + 1)a^5 + (4b^2 + 9)a^4 + (13b^3 + 23b^2 + 16b + 13)a^3 + (9b^4 + 30b^3 \\
&+ 78b^2 - 4b + 4)a^2+ (b^5 - 6b^4 + 82b^3 + 48b^2 - 33b - 1)a\\
& - 12b^5 + 24b^4 + 26b^3 + 2b^2 - 14b\big\}t\\
&+(a^2 + ab + b^2 + a - b + 1)^3.
\end{align*}
Assuming the coefficient of $t^2$ is not zero, this quadratic has a root $t \in \mathbb{F}_p$ if and only if its discriminant
\begin{align*}
D=&(a^2 - 4b)(ab + b^2 + a + 3b + 1)^2\\
& \times (a^3 - 5a^2 b - 8ab^2 - b^3 + 8a^2 + 6ab - 5b^2 + 5a + 8b - 1)^2
\end{align*}
is a square in $\mathbb{F}_p$.  Since the quadratic is irreducible, the factor $a^2-4b$ is not a square.  It follows that $D \in \mathbb{F}_p$ is a square if and only if one of the squared factors is zero.  In that case there is a unique root $t \in \mathbb{F}_p$.  The second squared factor is exactly $B(a,b)^2$. \medskip

We show now that in the first squared factor, $F(a,b)=a(b+1)+b^2+3b+1$ is never $0$ (mod $p$) for quadratic factors of $\hat H_{7,p}(x)$.  Assume it is zero, for some quadratic factor.  For the factor $x^2-x+1$, $F(-1,1)=3$, so by Theorem \ref{thm:1} it suffices to assume $p \equiv 1, 6$ (mod $7$).  Also, $F(a,-1)=-1$, so we can assume $b \neq -1$ and $a=-\frac{b^2+3b+1}{b+1}$.  The resultant
\begin{align*}
&\textrm{Res}_x\left(x^2-\frac{b^2+3b+1}{b+1} x+b, f_7(x,t)\right)\\
& = \frac{(b^6 + (3 - t)b^5 + (6 + 4t)b^4 + (7 + 13t)b^3 + (6 + 9t)b^2 + (t + 3)b+1)^2}{(b+1)^6}\\
& = \frac{(b^6 f_7(-1/b,t))^2}{(b+1)^6}
\end{align*}
must be zero, where $t$ is a root of the previous resultant above.  This implies that $f_7(x,t)=0$ has a root $x=-1/b \in \mathbb{F}_p$.  On the other hand, setting $\phi(x)=\frac{-1}{x-1}$, we have
$$(x-1)^6f_7(\phi(x),t)=f_7(x,t),$$
so the roots of $f_7(x,t)$ are invariant under $\phi(x)$, a linear fractional map of order $3$.  Since $x =-1/b \neq 1$ and the fixed points of $\phi$ are primitive sixth roots of unity, this implies either that $b$ is a primitive cube root of unity or that $f_7(x,t)$ has at least three linear factors over $\mathbb{F}_p$.  The first possibility implies that $x = -1/b$ is a root of the irreducible quadratic $x^2+ax+b$, because
$$\frac{1}{b^2}-\frac{a}{b}+b = \frac{1}{b^2}+\frac{b^2+3b+1}{b(b+1)}+b = \frac{(b^2+b+1)^2}{b^2(b+1)}.$$
In the second case $q(x)=x^2+ax+b$ is a factor of $f_7(x,t)$; and its transform
$$\frac{1}{b}(x-1)^2q(\phi(x)) = x^2 + \frac{-2b - a}{b}x + \frac{b + 1 + a}{b}$$
would also be an irreducible factor of $f_7(x,t)$ distinct from $x^2+ax+b$.  Here, $b \neq \frac{b + 1 + a}{b}$, because $b^2-b-1+\frac{(b^2+3b+1)}{b+1} = \frac{b(b^2+b+1)}{b+1} \neq 0$.  But then $f(x,t)$ would have three linear factors and two irreducible quadratic factors, which is impossible. \medskip

This shows that, as long as the leading coefficient is nonzero, namely
$$-b(b + 1 + a)\big\{a^3 + (-5b + 8)a^2 + (-8b^2 - 43b + 5)a- b^3 - 54b^2 - 41b - 1\big\} \not \equiv 0 \ (\textrm{mod} \ p),$$
then an irreducible quadratic factor $x^2+ax+b$ of $\hat H_{7,p}(x)$ divides $f_7(x,t)$ for some value of $t \in \mathbb{F}_p$ if and only if $B(a,b) \equiv 0$ (mod $p$).  It is clear that $b(b+1+a) \not \equiv 0$ (mod $p$) (otherwise the quadratic is reducible), so it suffices to show that
$$C(a,b)=a^3 + (-5b + 8)a^2 + (-8b^2 - 43b + 5)a- b^3 - 54b^2 - 41b - 1$$
is never $0$ (mod $p$).  Note that the polynomial $C(a,b)$ factors over the field $\mathbb{Q}(r)$, where $r^3-8r^2+5r+1=0$:
\begin{align*}
C(a,b)=& \ ((r - 1)b + r^2 - 7r - 2 - a)((r^2 - 7r - 1)b + r^2 - 8r + 6 + a)\\
&\  \times ((r^2 - 8r + 5)b - r - a).
\end{align*}
Thus, $C(a,b) \equiv 0 \ (\textrm{mod} \ p)$ is equivalent to
$$a \equiv (r-1)b-r^\sigma \ \ (\textrm{mod} \ \mathfrak{p}),$$
for some root $r$ of $x^3-8x^2+5x+1=0$, where $\sigma \in \textrm{Gal}(\mathbb{Q}(r)/\mathbb{Q})$ satisfies
$$r^\sigma=-r^2+7r+2, \ \ r^{\sigma^2} = r^2 - 8r + 6$$
and $\mathfrak{p}$ is a first degree prime ideal factor of $p$ in the real field $\mathbb{Q}(r)=\mathbb{Q}(\zeta_7)^+$.  (Recall that primes $p \equiv 1, 6$ (mod $7$) split completely in $\mathbb{Q}(r)$.) But
$$\textrm{disc}(x^2+((r-1)b-r^\sigma)x+b)=((r-1)b-r^\sigma)^2-4b=(r-1)^2(b + 8r^2 - 57r - 10)^2$$
is then a square in $\mathbb{F}_p$, contrary to the assumption that $x^2+ax+b$ is irreducible over $\mathbb{F}_p$. \medskip

Finally, if a factor $x^2+ax+b$ of $\hat H_{7,p}(x)$ divides
$$f_7(x,t) = (x^2-x+1)^3-tx(x-1)(x^3-8x^2+5x+1)$$
 for $t \in \mathbb{F}_p$, then for a root $d$ of $x^2+ax+b$ it follows that
 $$t = \frac{(d^2-d+1)^3}{d(d-1)(d^3-8d^2+5d+1)} = j_7^*$$
 is supersingular for $\Gamma_0^*(7)$, by (\ref{eqn:0}).
\end{proof}

\newtheorem{prop}{Proposition}

The following proposition gives a means for computing the number of supersingular invariants $j_7^* \in \mathbb{F}_p$.

\begin{prop} If $j_7^* \in \mathbb{F}_p$ is supersingular for $\Gamma^*_0(7)$ and distinct from $0, -1, 27$ (mod $p$), then for $t = j_7^*$, the polynomial $f_7(x,t) \in \mathbb{F}_p[x]$ is either irreducible, splits into $6$ linear or $2$ cubic factors over $\mathbb{F}_p$, or splits into $3$ quadratic factors $x^2+ax+b$ over $\mathbb{F}_p$ satisfying $B(a,b) \equiv 0$ (mod $p$).  All of the irreducible factors of $f_7(x,t)$ divide $\hat H_{7,p}(x)$.  Conversely, every linear factor, every quadratic factor $x^2 + ax +b$ with $B(a,b) \equiv 0$ (mod $p$), and every cubic factor of $\hat H_{7,p}(x)$ divides $f_7(x,t)$ (mod $p$), for a unique value of $t = j_7^* \in \mathbb{F}_p$.
\label{prop:1}
\end{prop}

\begin{proof} We have
$$\textrm{disc}(f_7(x,t))=7^4 t^4 (t + 1)^3 (t - 27)^3,$$
so the assumption on $j_7^* = t$ implies that $f_7(x,t)$ has no multiple roots. \medskip

Note that the set of roots of the polynomial $f_7(x,t)$ is invariant under a group $\mathcal{G}_7$ of linear fractional transformations, which is isomorphic to $S_3$:
$$\mathcal{G}_7=\Big\{x, \phi(x), \phi^2(x), \frac{x - r}{(1 - r)x - 1}, \frac{x - r^\sigma}{(1 - r^\sigma)x - 1}, \frac{x - r^{\sigma^2}}{(1 - r^{\sigma^2})x - 1}\Big\}.$$
We have verified this for $\phi(x)$ above.  For the mapping $T(x) = \frac{x - r}{(1 - r)x - 1}$ we have
$$((1-r)x-1)^6 f_7(T(x),t) = \left(\frac{19}{7}r^2 - \frac{15}{7}r - \frac{1}{7}\right)(r+2)^3 f_7(x,t),$$
where $\eta=\frac{1}{7}(19r^2 - 15r - 1)$ is a unit in $\mathbb{Q}(r)$ and $N_\mathbb{Q}(r+2) = 7^2$.  Furthermore, the elements in the subgoup $\langle \phi \rangle$ fix the $j$-invariant
\begin{equation}
j_7(x) = \frac{(x^2-x+1)^3(x^6-11x^5+30x^4-15x^3-10x^2+5x+1)^3}{x^7(x-1)^7(x^3-8x^2+5x+1)},
\label{eqn:1}
\end{equation}
while the ``transposition'' $T(x)$ and its conjugates map $j_7(x)$ to the $j$-invariant
\begin{equation}
j_{7,7}(x)=\frac{(x^6 + 229x^5 + 270x^4 - 1695x^3 + 1430x^2 - 235x + 1)^3 (x^2 - x + 1)^3}{x(x-1)(x^3 - 8x^2 + 5x + 1)^7}
\label{eqn:2}
\end{equation}
of the isogenous curve $E_{7,7}=E_7/\langle (0,0) \rangle$\footnote{The equation of $ E_{7,7} $ is $ Y^2+(1+d-d^2)XY+7(d^2-d^3)Y=X^3 -d (d-1)(7d+6)X^2-6d(d-1)(d^5 -2d^4-7d^3 + 9d^2 -3d+1)X -d(d-1)(d^9-2d^8-34d^7+153d^6-229d^5+199d^4-111d^3+28d^2-7d+1)$.}.  (See \cite[pp. 15-17]{mor3}.)  It follows that if a root $x=d$ of $f_7(x,t)$ is a supersingular invariant for $E_7$, then so are all the roots of $f_7(x,t)$, since they are all images of $d$ by the mappings in $\mathcal{G}_7$ and $E_7(d)$ being supersingular implies that $E_{7,7}(d)$ is also. \medskip

It is clear that if $f_7(x,t)$ has a quadratic factor over $\mathbb{F}_p$ and does not have multiple roots, then it splits into irreducible quadratics, because of the action of the map $\phi(x)$.  The same goes when $f_7(x,t)$ has an irreducible cubic factor.  This holds because cubics of the form $g(x,a)=x^3+ax^2-(a+3)x+1$ are invariant under the substitution $x \rightarrow \phi(x)$, so the cofactor of $g(x,a)$ is also a cubic invariant under the same substitution and has the form $g(x,b)$, for some $b$.  Then $g(x,b)$ must also be irreducible, because in this case $p \equiv 3,5$ (mod $7$), when $\hat H_{7,p}(x)$ has no linear factors.  This implies that if $f_7(x, t)$ has a linear factor, it must split completely, since $p \equiv 6$ (mod $7$) and a quadratic factor would yield two additional quadratic factors under the action of $\phi(x)$. Alternatively, the mappings in $\mathcal{G}_7$ are then mappings in $GL(\mathbb{F}_p)$, since $x^3-8x^2+5x+1$ splits for primes $p \equiv 6$ (mod $7$).  For this reason, every linear factor $x-a$ of $\hat H_{7,p}(x)$ divides some polynomial $f_7(x,t)$ with $t = j_7^*$; namely, the polynomial
$$f_7(x,t) = \prod_{\sigma \in \mathcal{G}_7}{(x-\sigma(a))}.$$
The same assertion for quadratic and cubic factors follows from Lemma \ref{lem:0} and the previous discussion.
\end{proof}

We turn now to the proof that Conjectures \ref{conj:1}-\ref{conj:3} imply Nakaya's Conjecture 5 for $N=7$. \medskip

To count linear factors of $ss_p^{(7*)}(x)$ we have to account for multiple factors.  Once we have done that, we can appeal to Conjectures \ref{conj:1}-\ref{conj:3}.  So we compute
\begin{align}
\notag f_7(x,0) & = (x^2 - x + 1)^3,\\
\label{eqn:3} f_7(x,-1) & = (x^3 - x^2 - 2x + 1)^2,\\
\notag f_7(x,27) & = (x^3 - 15x^2 + 12x + 1)^2.
\end{align}
These values of $x$ correspond to $j$-invariants $j = 0, -15^3, 255^3$ and discriminants $D = -3, -7, -28$, respectively, and will occur as roots of $\hat H_{7,p}(x)$ for $p>7$ or $p=5$ if and only if $\left(\frac{-3}{p}\right)=-1$ or $\left(\frac{-7}{p}\right)=\left(\frac{p}{7}\right)=-1$. \medskip

Suppose now that $p \equiv 2, 4$ (mod $7$).  Then the discriminants $-7, -28$ are not supersingular for $p$, and in addition to $x^2-x+1$ we only have to count sextic factors of $\hat H_{7,p}(x)$.  By Conjecture \ref{conj:2}A, this number is $N_6$, and these correspond $1-1$ to roots $j_7^* \neq 0$ in $\mathbb{F}_p$ of $ss_p^{(7*)}(X)$.  Since we must add $\frac{1}{2}\left(1-\left(\frac{-3}{p}\right) \right)$ to $N_6$ to account for the discriminant $-3$, the total number of such roots is
\begin{align*}
L^{(7*)}(p)& =
\begin{cases}
 \frac{1}{2}h(-7p), \ \ p \equiv 1 \ (\textrm{mod} \ 8);\\
 h(-7p), \ \ \ p \equiv 5 \ (\textrm{mod} \ 8);\\
 \frac{1}{4}h(-7p), \ \ p \equiv 3 \ (\textrm{mod} \ 4).
\end{cases}
\end{align*}
These counts coincide with the values given by the formula
\begin{align*}
L^{(7*)}(p)=& \ \frac{1}{2}\left(1+\left(\frac{-p}{7}\right) \right) L(p)\\
& \ +\frac{1}{8}\bigg\{2+\left(1-\left(\frac{-1}{7p}\right) \right) \left(2+\left(\frac{-2}{7p}\right) \right)\bigg\} h(-7p)
\end{align*}
conjectured by Nakaya.  \medskip

Assume next that $p \equiv 3, 5$ (mod $7$).  In this case we must add the contributions of $x^2-x+1$ and sextic factors (Conjecture \ref{conj:2}B) to the contribution of cubic factors (Conjecture \ref{conj:1}B).  The two cubic factors of $f_7(x,t)$ listed in (\ref{eqn:3}) are irreducible and distinct in this case (their roots belong to $\mathbb{Q}(r)$ and they split mod $p$ if and only if $p \equiv 1, 6$ mod $7$), so each of these cubics corresponds to one value of $j_7^* \in \mathbb{F}_p$. This gives
\begin{align*}
L^{(7*)}(p)& =\frac{1}{2}(N_3-2) + 2 + N_6+\frac{1}{2}\left(1-\left(\frac{-3}{p}\right) \right)\\
& =\frac{1}{2} \cdot 2L(p)+N_6 + \frac{1}{2}\left(3-\left(\frac{-3}{p}\right) \right)\\
& = \begin{cases}
 L(p) + \frac{1}{2}h(-7p), \ \ p \equiv 1 \ (\textrm{mod} \ 8);\\
 L(p) + h(-7p), \ \ \ p \equiv 5 \ (\textrm{mod} \ 8);\\
 L(p)+ \frac{1}{4}h(-7p), \ \ p \equiv 3 \ (\textrm{mod} \ 4).
\end{cases}
\end{align*}

For $p \equiv 1$ (mod $7$), we only have to count quadratic factors.  Each value of $j_7^*$ in $\mathbb{F}_p$ will then correspond to three quadratics in Conjecture \ref{conj:3}A satisfying $B(a,b) \equiv 0$ (mod $p$), with the exception of $j_7^*=0$, which corresponds to the single quadratic $x^2-x+1$, when $\left(\frac{-3}{p}\right)=-1$.  Hence, we must subtract $1$ from the count $N_2$ in Conjecture \ref{conj:3}A in case $j=0$ is supersingular.  In all, we have
\begin{align*}
L^{(7*)}(p)& = \frac{1}{3}\bigg\{N_2-\frac{1}{2}\left(1-\left(\frac{-3}{p}\right) \right)\bigg\}+\frac{1}{2}\left(1-\left(\frac{-3}{p}\right) \right)\\
&=\begin{cases}
  \frac{1}{2}h(-7p), \ \ p \equiv 1 \ (\textrm{mod} \ 8);\\
 h(-7p), \ \ \ p \equiv 5 \ (\textrm{mod} \ 8);\\
 \frac{1}{4}h(-7p), \ \ p \equiv 3 \ (\textrm{mod} \ 4).
 \end{cases}
\end{align*}

Finally, consider a prime $p \equiv 6$ (mod $7$).  In this case the two exceptional cubics above split into linear factors mod $p$, so we must subtract $6$ from the count $N_1$ in Conjecture \ref{conj:1}A.  Each of these cubics contributes one value $j_7^* \in \mathbb{F}_p$, while six linear factors yield a single $j_7^*$.  Adding the number of quadratic factors from Conjecture \ref{conj:3}B gives that
\begin{align*}
& L^{(7*)}(p) = \frac{1}{6}(N_1-6)+2+\frac{1}{3}\bigg\{N_2-\frac{1}{2}\left(1-\left(\frac{-3}{p}\right) \right)\bigg\}+\frac{1}{2}\left(1-\left(\frac{-3}{p}\right) \right)\\
& = L(p) + 1 + \frac{1}{3}\bigg\{3a_p h(-7p)-3-\frac{3}{2}\left(1-\left(\frac{-3}{p}\right) \right)\bigg\}+\frac{1}{2}\left(1-\left(\frac{-3}{p}\right) \right);
\end{align*}
where
$$a_p=\frac{1}{8}\bigg\{2+\left(1-\left(\frac{-1}{7p}\right) \right) \left(2+\left(\frac{-2}{7p}\right) \right)\bigg\}.$$
Thus, we find that
$$L^{(7*)}(p) = L(p)+a_ph(-7p), \ \ p \equiv 6 \ (\textrm{mod} \ 7).$$
This completes the proof that Conjectures \ref{conj:1}-\ref{conj:3} imply Nakaya's Conjecture 5 for $N=7$.

\section{The modular function fields.}

We consider the sequence of congruence groups
$$\Gamma(7) \subset \Gamma_1(7) \subset \Gamma_0(7) \subset \Gamma_0^*(7)$$
and the corresponding sequence of modular function fields over $k=\mathbb{Q}(\zeta_7)$, where $\zeta_7$ is a primitive $7$-th root of unity:
$$\textsf{K}_{\Gamma(7)} \supset \textsf{K}_{\Gamma_1(7)} \supset \textsf{K}_{\Gamma_0(7)} \supset \textsf{K}_{\Gamma_0^*(7)}.$$
We have the following field degrees:
$$[\textsf{K}_{\Gamma(7)}: \textsf{K}_{\Gamma_1(7)}] = 7, [\textsf{K}_{\Gamma_1(7)}: \textsf{K}_{\Gamma_0(7)}]=3, [\textsf{K}_{\Gamma_0(7)}: \textsf{K}_{\Gamma_0^*(7)}]=2.$$
From Duke \cite{du}, we have
$$\textsf{K}_{\Gamma_1(7)}=k(h(\tau)), \ \ \textsf{K}_{\Gamma(7)}=k(h(\tau), s(\tau)),$$
where $h(\tau)$ and $s(\tau)$ are given by the infinite products
\begin{align*}
h(\tau) & = q^{-1} \prod_{n \ge 1}{\frac{(1-q^{7n-3})(1-q^{7n-4})(1-q^{7n-2})^2(1-q^{7n-5})^2}{(1-q^{7n-1})^3(1-q^{7n-6})^3}}\\
& = q^{-1} + 3 + 4q + 3q^2 - 5q^4 - 7q^5 - 2q^6 + 8q^7 + 16q^8 + 12q^9 + \cdots,
\end{align*}
and
\begin{align*}
s(\tau) & = q^{-3/7} \prod_{n \ge 1}{\frac{(1-q^{7n-3})(1-q^{7n-4})}{(1-q^{7n-1})(1-q^{7n-6})}}\\
& =q^{-3/7}(1+q+q^2-q^4-q^5+q^7+2q^8+q^9-q^{10}-3q^{11} +\cdots);
\end{align*}
and $(h,s) = (h(\tau),s(\tau))$ parametrizes the Klein curve
$$X(7): \ s^7=h(h-1)^2.$$
We shall also need the identity
$$\left(\frac{\eta(\tau)}{\eta(7\tau)}\right)^4=\frac{h^3-8h^2+5h+1}{h(h-1)}, \ \ h=h(\tau),$$
where $\eta(\tau)$ is the Dedekind $\eta$-function.  A proof of this identity is given in Theorem \ref{thm:10} in the Appendix.  Also see \cite{mor5}.  From Schertz \cite[pp. 46, 51]{sch} we know that
$$\left(\frac{\eta(\tau)}{\eta(7\tau)}\right)^4  = \frac{1}{q} - 4 + 2q + 8q^2 - 5q^3 - 4q^4 - 10q^5 + 12q^6 - 7q^7 + 8q^8 + 46q^9 + \cdots$$
is a Hauptmodul for $\Gamma_0(7)$.  (Also see \cite[(4.18), p. 88]{elk}, where this function is denoted by $j_7$.  The above identity is (4.24) on p. 89 of this paper, where Elkies uses $d$ for our $h$.)  Setting
$$z = z(\tau) = \left(\frac{\eta(\tau)}{\eta(7\tau)}\right)^4+8$$
we therefore have that $\textsf{K}_{\Gamma_0(7)}=k(z(\tau))$.  The field $k(h(\tau))$ is a cyclic cubic extension of $k(z(\tau))$, with Galois group 
$$G(k(h)/k(z)) = \{1, \phi, \phi^2\}, \ \ \phi: h \rightarrow \phi(h) = \frac {1}{1-h},$$
and we have
$$z = \frac{h^3-8h^2+5h+1}{h(h-1)} + 8 = \frac{h^3-3h+1}{h(h-1)}=h+\phi(h)+\phi^2(h) = \textsf{Tr}_{k(h)/k(z)}(h).$$

From \cite{du} we have that
\begin{equation*}
j(\tau) = j_7(d(\tau)), \ \ \textrm{with} \ \ d(\tau) = h\left(\frac{-1}{7\tau}\right),
\end{equation*}
is the $j$-invariant of $E_7(d(\tau))$, where $j_7(x)$ is the rational function in (\ref{eqn:1}).  Hence,
\begin{equation}
j(7\tau) = j\left(\frac{-1}{7\tau}\right) = j_7(h(\tau)).
\label{eqn:4}
\end{equation}
On the other hand, we also have from \cite[p. 156]{du}, \cite[Thm. 3]{mor5} that
\begin{equation}
d(\tau) = h\left(\frac{-1}{7\tau}\right) = \frac{h(\tau)-r_1}{(1-r_1)h(\tau)-1}=T_1(h(\tau)),
\end{equation}
where 
$$T_1(x) = \frac{x-r_1}{(1-r_1)x-1}$$
and $r_1 = \frac{\sin(\pi/7)\sin^2(3\pi/7)}{\sin^3(2\pi/7)} = 3(\zeta^5+\zeta^2) + \zeta^3 + \zeta^4 + 4$ (with $\zeta = e^{2 \pi i/7}$) is a root of
$$p(x)=x^3-8x^2+5x+1=0.$$
If we set
$$T_i(x) = \frac{x-r_i}{(1-r_i)x-1}, \ \ 1 \le i \le 3,$$
where $r_2, r_3$ are the other two roots of $p(x)$, it can be checked that the transformations $T_i$ have order $2$, and that
$$j_7(T_i(x)) = j_{7,7}(x), \ \ 1 \le i \le 3,$$
where $j_{7,7}(x)$ is defined by (\ref{eqn:2}).  Using $d(\tau) = T_1(h(\tau))$, we see that
\begin{equation}
j(\tau) = j_7(d(\tau)) =  j_7(T_1(h(\tau))) = j_{7,7}(h(\tau)).
\label{eqn:4.5}
\end{equation}
This gives that
\begin{equation}
j(\tau) = \frac{(h^2 - h + 1)^3 (h^6 + 229h^5 + 270h^4 - 1695h^3 + 1430h^2 - 235h + 1)^3}{h(h-1)(h^3 - 8h^2 + 5h + 1)^7}.
\label{eqn:5}
\end{equation}
It is well-known \cite[p. 206]{co} that
$$K_{\Gamma_0(7)}=k(j(\tau),j(7\tau)) = k(z(\tau)),$$
so that $j(\tau)$ and $j(7\tau)$ are rational functions of $z(\tau)$.  Using (\ref{eqn:4}), (\ref{eqn:5}), and $z = \frac{h^3-3h+1}{h(h-1)}$
we find easily that
\begin{align}
\label{eqn:6} j(7\tau) & = \frac{(z^2-3z+9)(z^2-11z+25)^3}{z-8},\\
\label{eqn:7} j(\tau) & = \frac{(z^2-3z+9)(z^2+229z+505)^3}{(z-8)^7}, \ \ z = z(\tau).
\end{align}
Also, note that the group $\mathcal{G}_7=\langle \phi, T_1 \rangle \cong S_3$ acting on $k(h)$ has the fixed field
$$\textsf{K}_{\mathcal{G}_7}=k\left(\frac{z^2-15}{z-8}\right),$$
since $z$ maps to $A(z) = \frac{8z-15}{z-8}$ under the mapping $h \rightarrow T_1(h)$, and $z+A(z)=\frac{z^2-15}{z-8}$.  Note that
\begin{align*}
z + A(z) = & \ \left(\frac{\eta(\tau)}{\eta(7\tau)}\right)^4+8 + \left\{8\left(\left(\frac{\eta(\tau)}{\eta(7\tau)}\right)^4+8\right)-15\right\}\left(\frac{\eta(7\tau)}{\eta(\tau)}\right)^4\\
= & \ \left(\frac{\eta(\tau)}{\eta(7\tau)}\right)^4+16+49\left(\frac{\eta(7\tau)}{\eta(\tau)}\right)^4\\
= & \ j_7^*(\tau)+3.
\end{align*}
Thus, $\textsf{K}_{\mathcal{G}_7}=k(j_7^*(\tau))= \textsf{K}_{\Gamma_0^*(7)}$.  It is clear that $j(\tau)$ and $j(7\tau)=j\left(\frac{-1}{7\tau}\right)$ are conjugate over $\textsf{K}_{\mathcal{G}_7}$ and therefore satisfy a quadratic equation over this field. \medskip

In summary, we have the following.

\begin{thm} The field $\textsf{K}_{\Gamma_1(7)} = k(h(\tau))$ is a normal extension of $\textsf{K}_{\mathcal{G}_7}=\textsf{K}_{\Gamma_0^*(7)} = k(j_7^*(\tau))$ with Galois group $\mathcal{G}_7 \cong S_3 \cong D_3$.
\label{thm:2}
\end{thm}

Setting $z'=A(z)$, and using the relations
$$z+z'=j_7^*+3, \ \ z z' = 8j_7^*+9,$$
we find that the minimal polynomial of $h(\tau)$ over $\textsf{K}_{\mathcal{G}_7}$ is the $6$-th degree polynomial
\begin{align*}
P(x)=& \ (x^3-z x^2 + (z-3)x+1)(h^3-z' x^2 +(z'-3)x+1)\\
= & \ x^6-(j_7^*+3)x^5+(9j_7^*+6)x^4-(13j_7^*+7)x^3\\
& +(4j_7^*+6)x^2+(j_7^*-3)x+1\\
= & \ f_7(x,j_7^*),
\end{align*}
where the polynomial $f_7(x,t)$ is defined in Conjecture \ref{conj:2}. \medskip

Set
\begin{align*}
G(x,j) = & (x^2 - x + 1)^3 (x^6 + 229x^5 + 270x^4 - 1695x^3 + 1430x^2 - 235x + 1)^3\\
&\ - jx(x-1)(x^3 - 8x^2 + 5x + 1)^7.
\end{align*}
Using the fact that $f_7(h(\tau),j_7^*(\tau)) = 0$ and $G(h(\tau),j(\tau)) = 0$ from (\ref{eqn:5}), the resultant 
\begin{equation*}
\textrm{Res}_h(G(h,X),f_7(h,Y)) = 7^{42} R_7(X,Y)^3
\end{equation*}
gives the relation
\begin{align*}
R_7(X,Y) &= X^2-XY(Y^2-21Y+8)(Y^4-42Y^3+454Y^2-1008Y-1280)\\
& \ \ +Y^2(Y^2+224Y+448)^3
\end{align*}
satisfied by $X = j(\tau)$ and $Y = j_7^*(\tau)$.

\section{Class fields over $K=\mathbb{Q}(\sqrt{-l})$.}

In this section we only consider primes $l$ satisfying $l \equiv 3, 5, 6$ (mod $7$).

\begin{lem}
\begin{enumerate}[(i)]
\item Let $l>3$ be a prime with $\left(\frac{-l}{7}\right)=+1$ and $d = l$ or $d = 4l$ in the cases $l \equiv 3$ and $l \equiv 1$ (mod $4$), respectively.  If $w=\frac{v+\sqrt{-d}}{2}$ is a primitive algebraic integer in $K = \mathbb{Q}(\sqrt{-l})$ satisfying the condition that $N(w) = \frac{v^2+d}{4} \equiv 0$ (mod $7^2$), then $\left(\frac{\eta(w/7)}{\eta(w)}\right)^4 \in \Sigma$, the Hilbert class field of $K$.

\item If $l \equiv 3$ (mod $4$),$\left(\frac{-l}{7}\right)=+1$ and $w = v + \sqrt{-l}$ satisfies $N(w) = v^2+l \equiv 0$ (mod $7^2$) and $2 \nmid N(w)$, then $\left(\frac{\eta(w/7)}{\eta(w)}\right)^4 \in \Omega_2$, the ring class field of conductor $f = 2$ over $K$.
\end{enumerate}
\label{lem:1}
\end{lem}

\begin{proof} (i) By Schertz's Theorem 6.6.4 \cite[p. 159]{sch}, with $\gamma_2(\tau) = j^{1/3}(\tau)$ and $\gamma_3(\tau) = \sqrt{j(\tau)-1728}$, the quantities
\begin{align*}
\alpha =  & \ \left(\frac{\eta(w/7)}{\eta(w)}\right)^8 \gamma_2(w)^6 = \left(\frac{\eta(w/7)}{\eta(w)}\right)^8 j(w)^2,\\
\beta = & \ \left(\frac{\eta(w/7)}{\eta(w)}\right)^6 \gamma_3(w)^3 = \left(\frac{\eta(w/7)}{\eta(w)}\right)^6 (j(w)-12^3)^{3/2},
\end{align*}
lie in $\Sigma$, since $w/1$ and $w/7$ are basis quotients for the ideals $(1) = R_K$ and $\wp_7=(7,w)$ in $R_K$.  It follows that
$$\frac{\beta^2}{\alpha} = \left(\frac{\eta(w/7)}{\eta(w)}\right)^4 \frac{(j(w)-12^3)^3}{j(w)^2} \in \Sigma.$$
Since $l>3$, $j(w) \neq 0, 1728$ and the first assertion of the lemma follows.  The second follows in the same way, using that $w/1$ and $w/7$ are basis quotients for ideals in the ring $\textsf{R}_{-4l}$ of discriminant $d=-4l$ in $K$.
\end{proof}
 
\begin{thm} If $l$ is a prime for which $\left(\frac{-l}{7}\right) = +1$ and $w$ satisfies the conditions (i) or (ii) of Lemma 1, then the ring class field $\Omega_f$ ($f = 1$, resp. $2$) of the field $K=\mathbb{Q}(\sqrt{-l})$ is generated over $\mathbb{Q}$ by
\begin{equation*}
z(w/7) = \left(\frac{\eta(w/7)}{\eta(w)}\right)^4+8.
\end{equation*}
\label{thm:3}
\end{thm}

\begin{proof} Letting $\tau=w/7$, Lemma \ref{lem:1} shows that $z(\tau) \in \Omega_f$, where $f=1$ or $2$, resp., in cases (i) and (ii) of the Lemma.  The claim is that $z(\tau)$ generates $\Omega_f$ over $\mathbb{Q}$.  On the one hand, it is clear that $j(\tau) \in \mathbb{Q}(z(\tau))$ by equation (\ref{eqn:7}).  Since $z(\tau) \in \Omega_f$, it follows that $z(\tau)$ has degree $h(-d)$ or $2h(-d)$ over $\mathbb{Q}$, where $d = f^2l$.  Suppose the former holds.  Now the real value $j(w)=j(7\tau)$ is a rational function of $z(\tau)$, by (\ref{eqn:6}).  If $[\mathbb{Q}(z(\tau)):\mathbb{Q}] = h(-d)$, then $\mathbb{Q}(j(w)) = \mathbb{Q}(z(\tau))$.  From (\ref{eqn:7}) it follows that $j(w/7)$ is a real number, whence it follows that $\wp_7^2 \sim 1$ must be principal in the ring class group $\mathcal{C}_{\mathcal{O}}$ of $\mathcal{O} = \textsf{R}_{-f^2 l} \subset K$.  Hence, we have a primitive solution $(x,y)$ of the equation
$$4 \cdot 7^2 = x^2 + lf^2y^2, \ \ 0 \le x \le 13.$$
This implies that $l $ is one of the primes $3, 5, 13, 19$.  In these cases the minimal polynomial $m(x)$ of $z(\tau)$ over $\mathbb{Q}$ divides both polynomials
$$F_d(x) = (x-8)^{h(-d)}H_{-d}\left(\frac{(x^2-3x+9)(x^2-11x+25)^3}{x-8}\right)$$
and
$$G_d(x) = (x-8)^{7h(-d)}H_{-d}\left(\frac{(x^2-3x+9)(x^2+229x+505)^3}{(x-8)^7}\right),$$
by (\ref{eqn:6}) and (\ref{eqn:7}), where $H_{-d}(X)$ is the class equation for the discriminant $-d = d_K f^2 = \textrm{disc}(K) f^2= -l$ or $-4l$.  We have
\begin{align*}
& H_{-3}(X) = X,\\
& H_{-12}(X) = X-54000,\\
& H_{-20}(X) = X^2 - 1264000X - 681472000,\\
& H_{-52}(X) = X^2 - 6896880000X - 567663552000000,\\
& H_{-19}(X) = X+96^3.
\end{align*}
Now comparing the factorizations of $F_d(x)$ and $G_d(x)$ yields
\begin{align*}
l=3: \ & m(x) = x^2-3x+9;\\
l = 3, f = 2: \ & m(x) = x^2-5x+25;\\
l = 5: \ & m(x) = x^4 - 2x^3 - 9x^2 - 50x + 305;\\
l = 13: \ & m(x) = x^4 - 58x^3 + 1599x^2 - 17770x + 67825\\
l = 19: \ & m(x) = x^2 - 11x + 73.
 \end{align*}
 Thus, $z(\tau)$ has degree $2h(-d)$ in every case.  This proves the theorem.
 \end{proof}
 
 \begin{lem} If $\tau=w/7$, where $w = \frac{v+\sqrt{-d}}{2}$ is divisible by $\wp_7^2$, the quantities
\begin{equation}
z-8 = \left(\frac{\eta(\tau)}{\eta(7\tau)}\right)^4\cong \wp_7'^2 \ \textrm{and} \ z_1 - 8 = 49\left(\frac{\eta(7\tau)}{\eta(\tau)}\right)^4\cong \wp_7^2
\label{eqn:8}
\end{equation}
are conjugates over $\mathbb{Q}$.
\label{lem:2}
\end{lem}

\begin{proof} First note that the ideal factorizations $z-8 \cong \wp_7'^2$ and $z_1 - 8 \cong \wp_7^2$ in (\ref{eqn:8}) follow from a classical result \cite[p.43, (1.3.)]{d2}.  We have
$$z - 8 = z(w/7) - 8 = \left(\frac{\eta(w/7)}{\eta(w)}\right)^4$$
and
\begin{equation*}
z_1-8 = 49\left(\frac{\eta(w)}{\eta(w/7)}\right)^4 = \frac{49}{z-8} = z\left(\frac{-7}{w}\right) -8,
\end{equation*}
by the transformation formula $\eta(-1/\tau) = \sqrt{\frac{\tau}{i}} \eta(\tau)$.  Furthermore, if $w', \mathfrak{a}'$ denote the conjugates of quantities (or ideals) $w, \mathfrak{a}$ in $K$,
$$\frac{-7}{w} = \frac{-7w'}{N(w)} = \frac{-w'}{7a} = \frac{-v+\sqrt{-d}}{2(7a)},$$
where
\begin{equation*}
w = \frac{v+\sqrt{-d}}{2}, \ \ N(w) = 7^2a.
\end{equation*}
Setting $\mathfrak{a} = (a, w)$, we have $\wp_7 \mathfrak{a} = (7a,w)$ and 
$$z_1 - 8 = z\left(\frac{-w'}{7a}\right) - 8 = \overline{z\left(\frac{w}{7a}\right)} - 8.$$
Putting $J(x)$ equal to the rational function on the right side of (\ref{eqn:6}), we have
$$J(z(w/7)) = J(z) = \frac{(z^2-3z+9)(z^2-11z+25)^3}{z-8} = j(w),$$
and
$$J(z_1) = J(z(-w'/7a)) = j(-w'/a) = j(\mathfrak{a}') = \overline{j(\mathfrak{a})}.$$
Thus, $J(\bar z_1) = j(\mathfrak{a})$ and if $j(\mathfrak{a})^\sigma = j(w)$, with $\sigma \in \textrm{Gal}(\Omega_f/K)$, then
$$J(z^{\sigma^{-1}}) = j(\mathfrak{a}) = J(\bar z_1).$$
I claim that $\bar z_1 = z^{\sigma^{-1}}$.  Note that $\bar z_1 -8 \cong z^{\sigma^{-1}} -8 \cong z-8 \cong \wp_7'^2$. \medskip

Assume that $u=\bar z_1-8$ is distinct from $v=z^{\sigma^{-1}}-8$.  Then $(u,v)$ is a point on the curve $F(u,v)=0$, where

\begin{align*}
F(u,v) & = u v \frac{J(u+8)-J(v+8)}{u-v}\\
& = u^7 v + (v^2 + 28v)u^6 + (v^3 + 28v^2 + 322v)u^5\\
& \ + v(v^3 + 28v^2 + 322v + 1904)u^4\\
& \ + v(v^4 + 28v^3 + 322v^2 + 1904v + 5915)u^3\\
& \ + v(v^5 + 28v^4 + 322v^3 + 1904v^2 + 5915v + 8624)u^2\\
& \ + v(v^6 + 28v^5 + 322v^4 + 1904v^3 + 5915v^2 + 8624v + 4018)u-49.
\end{align*}
If we set
$$A_k = A_k(u,v) = \sum_{i=1}^{k-1}{u^i v^{k-i}},$$
then the above expression gives that
\begin{align*}
F(u,v) & = A_8+28A_7+322A_6+1904 A_5+5915A_4+8624A_3+4018A_2-49,
\end{align*}
where all coefficients except the first are divisible by $7$.  Moreover, since $u \cong v$, $u^k \mid A_k$, whence $\wp_7'^{2k} \mid A_k$.  But then $F(u,v)=0$ would imply $\wp_7'^4 \mid 49$ in $\Omega_f$, which is not the case.  This contradiction shows that $\bar z_1 = z^{\sigma^{-1}}$, and therefore $z_1$ is a conjugate of $z$ over $\mathbb{Q}$. 
\end{proof}

This lemma implies the following result.

\begin{thm} If $m_d(x)$ is the minimal polynomial of $z(\tau) = z(w/7)$ over $\mathbb{Q}$, the polynomial $P_d(x)$ defined by
$$P_d(x) = x^{2h}(x-1)^{2h}m_d\left(\frac{x^3-3x+1}{x(x-1)}\right)$$
is irreducible over $\mathbb{Q}$ and is the minimal polynomial of $h(\tau) = h(w/7)$ over $\mathbb{Q}$.  Furthermore,
$$((1-r)x-1)^{6h(-d)} P_d(T_1(x)) = (r^2-r)^{2h(-d)} 7^{4h(-d)} P_d(x).$$
\label{thm:4}
\end{thm}

\begin{proof}
From Section 4, we know that $h(\tau)$ satisfies the cubic equation
\begin{equation}
x^3-3x+1-z(\tau)x(x-1)=0.
\label{eqn:9}
\end{equation}
I claim this equation is irreducible over $\Omega_f = \mathbb{Q}(z(\tau))$, for $l > 3$.  If this equation factors, then it is a product of three linear factors over $\Omega_f$ (by the action of $\phi$), so its norm $N_\mathbb{Q}(x^3-3x+1-z(\tau)x(x-1))$ to $\mathbb{Q}$ factors as a product of three factors of degree $2h(-d)$ (with $d = f^2l$, as before).  We have
$$N_\mathbb{Q}(x^3-3x+1-z(\tau)x(x-1)) = x^{2h}(x-1)^{2h}m_d\left(\frac{x^3-3x+1}{x(x-1)}\right)=P_d(x).$$
Now apply the mapping $T_1(x)$ to $P_d(x)$.  With $r=r_1$ we obtain:
\begin{align*}
((1-r)x-1)^{6h}P_d(T_1(x)) = &\\
((1-r)x-1)^{6h} &T_1(x)^{2h} (T_1(x)-1)^{2h}m_d\left(\frac{T_1(x)^3-3T_1(x)+1}{T_1(x)(T_1(x)-1)}\right).
\end{align*}
Let $L=\mathbb{Q}(r)$.  Then $L \cap \Omega_f = \mathbb{Q}$, since $L/\mathbb{Q}$ is totally ramified at $p=7$.  Now, letting $z(\tau) = \eta+8$, we have
\begin{align*}
T_1(x)^{2h} (T_1(x)-1)^{2h}&m_d\left(\frac{T_1(x)^3-3T_1(x)+1}{T_1(x)(T_1(x)-1)}\right) =\\
&N_L\left(T_1(x)^3-3T_1(x)+1-(\eta+8)T_1(x)(T_1(x)-1)\right);
\end{align*}
here $N_L$ denotes the norm from $\Omega_f L$ to $L$.  We compute
\begin{align*}
T_1(x)^3-3T_1(x)+&1-(\eta+8)T_1(x)(T_1(x)-1)\\
&= r(r-1)\frac{\eta x^3 + (-8\eta - 49)x^2 + (5\eta + 49)x + \eta}{((1-r)x-1)^3} \\
& = r(r-1) \eta \frac{x^3 - \left(8 + \frac{49}{\eta}\right)x^2 + \left(5 + \frac{49}{\eta}\right)x + 1}{((1-r)x-1)^3}.
\end{align*}
Thus,
\begin{align*}
((1-r)x-1)^{6h}&P_d(T_1(x))\\
= (r(r-1))^{2h} &N_L(\eta) N_L\left(x^3 - \left(8 + \frac{49}{\eta}\right)x^2 + \left(5 + \frac{49}{\eta}\right)x + 1\right).
\end{align*}
Since $\eta$ and $49/\eta$ are conjugate over $\mathbb{Q}$, it follows that
\begin{align*}
((1-r)x-1)^{6h}& P_d(T_1(x)) \\
 & = (r(r-1))^{2h} 7^{4h} N_L(x^3 - (8 + \eta)x^2 + (5 + \eta)x + 1)\\
 & = (r^2-r)^{2h} 7^{4h} P_d(x).
\end{align*}

Now, if $P_d(x)$ factors over $\mathbb{Q}$, it is a product of three factors of degree $2h=2h(-d)$ over $\mathbb{Q}$, and the roots of each factor lie in $\Omega_f$.  But the above computation shows that if $\alpha$ is any root of $P_d(x)$, $T_1(\alpha)$ is also a root.  On the other hand,
$$T_1(\alpha) = \frac{\alpha-r}{(1-r)\alpha-1} = \frac{-r+\alpha}{-\alpha r+\alpha-1}$$
is a linear fractional expression in $r$ with determinant $\alpha^2-\alpha+1\neq 0$, since $l > 3$ and $\mathbb{Q}(\alpha) = \Omega_f$. But then $\alpha, T_1(\alpha) \in \Omega_f$ would imply that $r \in \Omega_f$, which is false.
\end{proof}

We now prove the following.

\begin{thm} If $d= l$ or $4l$, $\left(\frac{-l}{7}\right) = +1$, and $w = \frac{v+\sqrt{-d}}{2}$ is divisible by $\wp_7^2$, then 
$$\mathbb{Q}(h(w/7)) = \Sigma_{\wp_7'} \Omega_f,$$
where $\Sigma_{\wp_7'}$ is the ray class field of conductor $\mathfrak{f}=\wp_7'$ over $K$.  Hence, $h(w/7)$ generates the inertia field of $\wp_7$ in $\Sigma_7 \Omega_f$ over $K = \mathbb{Q}(\sqrt{-l})$.
\label{thm:5}
\end{thm}

\begin{proof}
First, $h(\tau)$, for $\tau = w/7$, lies in $\Sigma_7 \Omega_2$, by a result of S\"ohngen \cite{sohn}.  See also \cite{cho} and \cite[p. 315]{co}.  Now, $[\Sigma_7\Omega_f:\Omega_f] = [\Sigma_7:\Sigma] = \frac{1}{2} \varphi(\wp_7) \varphi(\wp_7') = 18$ and $[\Sigma_{\wp_7'}\Omega_f:\Omega_f] = [\Sigma_{\wp_7'}:\Sigma] = \frac{1}{2}\varphi(\wp_7') = 3$.  Note that the discriminant of the cubic in (\ref{eqn:9}) is
$$\textrm{disc}(x^3-3x+1-z(w/7)x(x-1)) = (z^2-3z+9)^2, \ \ z = z(w/7).$$
Furthermore,
$$z^2-3z+9 = (z-8)^2+13(z-8)+49=(z-8)(z+5)+49.$$
Since $z-8 = \left(\frac{\eta(w/7)}{\eta(w)}\right)^4 \cong \wp_7'^2$, it follows that $z^2-3z+9 \cong \wp_7'^2 \mathfrak{a}$, for some ideal $\mathfrak{a}$ in $R_{\Omega_f}$.  Furthermore, by (\ref{eqn:7}),
$$j(\tau) \equiv \frac{(z+5)(z+6)^7}{(z+6)^7} \equiv z+5 \ (\textrm{mod} \ \mathfrak{p}), \ \ \mathfrak{p} \mid \wp_7 \ \textrm{in} \ \Omega_f.$$
Hence, the minimal polynomial $m_d(x)$ of $z=z(\tau)$ over $K$ satisfies
$$m_d(x) \equiv H_{d_K}(x+5) \ (\textrm{mod} \ \wp_7).$$
It follows that $m_d(x)$ factors modulo $\wp_7$ into polynomials of degree $\textrm{ord}(\wp_7)$, the order of $\wp_7$ in the ring class group $\mathcal{C}_f$ mod $f$ of $K$ corresponding to the order $\mathcal{O} = \textsf{R}_{-d}$ of discriminant $-d = d_K f^2$.  This is because $7$ does not divide the discriminant of $H_{-d}(x)$, by Deuring's results \cite{d4}. If $z+5$ is divisible by a prime divisor $\mathfrak{p}$ of $\wp_7$ in $\Omega_f$, then $j(\tau) \equiv 0$ (mod $\mathfrak{p}$) would imply that $\textrm{ord}(\wp_7)=1$, so that $\wp_7 \sim 1$ in $\mathcal{C}_f$.  Then $4\cdot 7 = x^2+ly^2$ implies $l=19$ (assuming $l>3$).  However, in the latter case, $j(\tau) = -96^3$, which is not divisible by $7$.  Therefore, if $l >3$, $(z+5)$ is relatively prime to $\wp_7$, so that $\wp_7 \nmid \mathfrak{a}$.  \medskip

It follows that $\Omega_f(h(\tau))/\Omega_f$ is unramified over $\wp_7$ and ramified over $\wp_7'$ (since some prime divisor of $7$ must ramify in the subfield $\Omega_f(h(\tau)) \subset \Sigma_7 \Omega_f$).  Thus, the inertia field of $\wp_7$ over $\mathbb{Q}$ contains $\Omega_f(h(\tau))$ and has degree at least $3$ over $\Omega_f$.  However, the prime divisors of $7$ in $\Sigma_7\Omega_f$ have ramification index $\ge 6$ over $\Omega_f$, since $\mathbb{Q}(\zeta_7) \subset \Sigma_7$.  From $[\Sigma_7\Omega_f: \Omega_f] = 18$ it follows that $\wp_7$ is purely ramified in $\Sigma_7\Omega_f/\Sigma_{\wp_7'}\Omega_f$.  Hence, the inertia field of $\wp_7$ in $\Sigma_7\Omega_f$ equals $\Sigma_{\wp_7'}\Omega_f$, and consequently $\Omega_f(h(\tau)) = \Sigma_{\wp_7'}\Omega_f$.  This proves the theorem.
\end{proof}

\begin{thm} If $d= l$ or $4l$, $\left(\frac{-l}{7}\right) = +1$, and $w = \frac{v+\sqrt{-d}}{2}$ is divisible by $\wp_7^2$, then 
$$\mathbb{Q}(h(-1/w)) = \Sigma_{\wp_7} \Omega_f.$$
\label{thm:6}
\end{thm}

\begin{proof} We use the relation $h(\frac{-1}{7\tau}) = T_1(h(\tau))$ with $\tau = w/7$ to deduce that
\begin{equation}
h(-1/w) = T_1(h(w/7)) = \frac{h(w/7)-r_1}{(1-r_1)h(w/7)-1}.
\label{eqn:9.5}
\end{equation}
Furthermore, Theorem \ref{thm:4} shows that
$$P_d(h(-1/w)) = P_d(T_1(h(w/7))) = P_d(h(w/7)) = 0.$$
Hence, the map $\sigma: h(w/7) \rightarrow h(-1/w)$ induces an isomorphism of $\Sigma_{\wp_7'}\Omega_f = \mathbb{Q}(h(w/7))$ with $\mathbb{Q}(h(-1/w))$.  The field $\Omega_f$ is normal over $\mathbb{Q}$, as is $K = \mathbb{Q}(\sqrt{-l})$, which implies that $\sigma$ either takes $\Sigma_{\wp_7'}$ to itself or to $\Sigma_{\wp_7}$.  However,
$$h(-1/w) = \frac{-r_1+h(w/7)}{-h(w/7)r_1+h(w/7)-1}$$
is linear fractional in $r_1$ with determinant $h^2(w/7)-h(w/7)+1 \neq 0$, so that $r_1 \notin \Sigma_{\wp_7'}\Omega_f$ implies that $h(-1/w) \notin \Sigma_{\wp_7'}\Omega_f$.  This implies that  $\mathbb{Q}(h(-1/w)) = \Sigma_{\wp_7} \Omega_f$, by Theorem \ref{thm:5}.
\end{proof}

From the end of Section 4 we know that $f_7(h(\tau),j_7^*(\tau)) = 0$.  Solving for $j_7^*(\tau)$ in this equation gives
$$j_7^*(\tau) = \frac{(h^2-h+1)^3}{h(h-1)(h^3-8h^2+5h+1)} = \Phi(h), \ \ h = h(\tau),$$
and therefore
$$j_7^*(w/7) =  \frac{(\eta^2-\eta+1)^3}{\eta(\eta-1)(\eta^3-8\eta^2+5\eta+1)} = \Phi(\eta), \ \ \eta = h(w/7).$$
Since the expression $\Phi(\eta)$ is invariant under the substitution $\eta \rightarrow \phi(\eta) = \frac{1}{1-\eta}$, which represents a non-trivial automorphism of $\Sigma_{\wp_7'} \Omega_f/\Omega_f$, it follows that
$$j_7^*(w/7) \in \Omega_f.$$
$\Phi(\eta)$ is also invariant under the isomorphism
$$\psi: h(w/7) \rightarrow h(-1/w) = T_1(h(w/7)),$$
which maps $\Omega_f$ to itself and takes $z(w/7)$ to
\begin{equation}
\psi(z(w/7)) = A(z(w/7)) = \frac{8z(w/7)-15}{z(w/7)-8}.
\label{eqn:n1}
\end{equation}
The fixed points of the linear fractional map $A(z)$ are $z = 1, 15$, so $A(z(w/7)) \neq z(w/7)$, by Theorem \ref{thm:3}.  The mapping $\psi$ therefore induces an automorphism of order $2$ on $\Omega_f$.  This implies that $j_7^*(w/7)$ lies in a field $L_f \subset \Omega_f$ over which $\Omega_f$ has degree $2$. \medskip

Note that if $l \equiv 6$ (mod $7$), then
$$\sqrt{-l} \equiv \pm 1 \ (\textrm{mod} \ \wp_7'),$$
so that $\mathfrak{l} =(\sqrt{-l})$ lies in the principal ray class (mod $\wp_7'$).  It also lies in the principal ring class mod $2$ if $l \equiv 3$ (mod $4$).  Therefore, it splits completely in $\Sigma_{\wp_7'}\Omega_f/K$.  Hence, the minimal polynomial $P_d(x)$ of $h(w/7)$ over $\mathbb{Q}$ splits into linear factors modulo $\mathfrak{l}$. \medskip

On the other hand, if if $l \equiv 3,5$ (mod $7$), then $-l \equiv 2, 4$ (mod $7$) and
$$\sqrt{-l} \equiv \pm 3, \pm 2 \ (\textrm{mod} \ \wp_7').$$
Since $3^3 \equiv -1$ and $2^3 \equiv 1$ (mod $7$), the ideal $\mathfrak{l}$ has order $3$ in the ray class group mod $\wp_7'$. Hence, $\mathfrak{l}$ splits into primes of relative (and absolute) degree $3$ in $\Sigma_{\wp_7'}\Omega_f/K$, while it splits completely in $\Omega_f/K$. \medskip

In either case, the prime divisors $\mathfrak{q}$ of $l$ in the extension $L_f/\mathbb{Q}$ have degree 1.  This yields the following.  (In this theorem $R_{L}$ denotes the ring of integers in the field $L$.)

\begin{thm} Let $l > 3$ be a prime satisfying $\left(\frac{-l}{7}\right) = +1$.  If $\mathfrak{q}$ is any prime divisor of $l$ in the field $L_f$ which is the fixed field of the automorphism
$\psi |_{\Omega_f}$ defined by (\ref{eqn:n1}), the reduced value $j_7^*(w/7)$ mod $\mathfrak{q}$ is a supersingular invariant for the Fricke group $\Gamma_0^*(7)$ lying in $\mathbb{F}_l = R_{L_f}/\mathfrak{q}$.
\label{thm:n1}
\end{thm}

\begin{proof} This follows from the above discussion and the fact that
$$R_7(j(w/7), j_7^*(w/7)) = 0 \ \Rightarrow \ R_7(j(w/7), j_7^*(w/7)) \equiv 0 \ (\textrm{mod} \ \mathfrak{p}),$$
where $\mathfrak{p} \mid \mathfrak{q}$ in $\Omega_f$.  Since $\mathfrak{p} \mid l$ and $l$ is ramified in $K$, the reduced value $j(w/7)$ mod $\mathfrak{p}$ is a root of $ss_p(X)$ in $\mathbb{F}_l = R_{\Omega_f}/\mathfrak{p} \cong R_{L_f}/\mathfrak{q}$.  The definition in the introduction shows that $j_7^*(w/7)$ mod $\mathfrak{q}$ must be a supersingular invariant for $\Gamma_0^*(7)$.
\end{proof}

It can be shown that when $l = 41$ and $w = 29 + \sqrt{-41}$, the minimal polynomial of $j_7^*(w/7) \in L_1 \subset \Omega_1 = \Sigma$ over $\mathbb{Q}$ is
\begin{align*}
\Psi_7(x) & = x^8 - 464x^7 + 125056x^6 - 5914240x^5 + 78235776x^4 + 38845440x^3\\
& \ \  + 527245312x^2 + 14082048x + 20123648\\
& \equiv (x + 1)(x + 14)(x + 8)^2(x + 29)^2(x + 31)^2 \ (\textrm{mod} \ 41).
\end{align*}
On the other hand (see \cite[Table 3]{mor3}),
\begin{align*}
ss_l&^{(7*)}(Y) \equiv Y(Y + 1)(Y + 8)(Y + 12)(Y + 13)(Y + 14)(Y + 17)(Y + 29)\\
& \times (Y + 31)(Y + 33)(Y + 39)(Y^2 + Y + 18)(Y^2 + 37Y + 26) \ \ (\textrm{mod} \ 41).
\end{align*}
Hence, the reduced values of $j_7^*(w/7)$ (mod $\mathfrak{q}$) in Theorem \ref{thm:n1} do not account for all of the supersingular values $j_7^* \in \mathbb{F}_l$.  In the sequel we will see that the remaining supersingular values in $\mathbb{F}_l$ arise as reductions of values of $j_7^*(\tau)$, where $\tau \in K' = \mathbb{Q}(\sqrt{-7l})$.

\section{Proof of Conjecture 1.}

\begin{prop} Assume $l > 7$ is a prime with $l \equiv 6$ (mod $7$). \smallskip

\noindent (a) If $l \equiv 2$ (mod $3$), the polynomial
\begin{equation}
\label{eqn:10} f_0(x) = (x^2-x+1)(x^6-11x^5+30x^4-15x^3-10x^2+5x+1)
\end{equation}
has exactly $6$ distinct linear factors modulo $l$. \smallskip

\noindent (b) If $l \equiv 3$ (mod $4$), the polynomial
\begin{align} 
\label{eqn:11} f_{1728}(x) =& x^{12}-18x^{11}+117x^{10}-354x^9+570x^8-486x^7+273x^6\\
\notag & -222x^5+174x^4-46x^3-15x^2+6x+1
\end{align}
has exactly $6$ distinct linear factors modulo $l$.
\label{prop:2}
\end{prop}

\begin{proof} (a) A root of the $6$-th degree factor $s(x)$ of $f_0(x)$ generates the real subfield $F^+ \subset F = \mathbb{Q}(\zeta_{21})$ of $21$-st roots of unity, since $s(\alpha) = 0$ with
$$\alpha = \xi^5 - \xi^4 - 5\xi^3 + 4\xi^2 + 4\xi, \ \ \xi = e^{2\pi i/21}+e^{-2\pi i/21}.$$
Note that the minimal polynomial of $\xi$ is
$$m(x) = x^6 - x^5 - 6x^4 + 6x^3 + 8x^2 - 8x + 1.$$
Now since $l \equiv 20$ (mod $21$), $l$ splits completely in $F^+$, and since
$$\textrm{disc}(x^6-11x^5+30x^4-15x^3-10x^2+5x+1) = 2^{12} \cdot 3^3 \cdot 7^5,$$
the polynomial $s(x)$ splits into distinct linear factors (mod $l$).  On the other hand, $x^2-x+1$ remains irreducible mod $l$, since $\left(\frac{-3}{l}\right) = -1$.  This proves the assertion. \smallskip

\noindent (b) First note that
\begin{align*}
f_{1728}(x) &= x^4(x-1)^4 q\left(\frac{x^3-3x+1}{x(x-1)}\right),\\
q(z) & = z^4 - 18z^3 + 111z^2 - 298z + 393.
\end{align*}
We use the fact that
$$q(z) = (z^2 - 9z + 29)^2-28(z-4)^2.$$
Putting $z=\frac{x^3-3x+1}{x(x-1)}$ and multiplying by $x^4(x-1)^4$ in the last expression gives that
\begin{align*}
f_{1728}(x) &= (x^6 - 9x^5 + 32x^4 - 29x^3 + 2x^2 + 3x + 1)^2\\
& \ \ \ -28x^2(x-1)^2(x^3 - 4x^2 + x + 1)^2\\
&= A(x)^2 -28x^2(x-1)^2B(x)^2;
\end{align*}
where the roots of the polynomials
\begin{align*}
A(x) & = x^6 - 9x^5 + 32x^4 - 29x^3 + 2x^2 + 3x + 1,\\
B(x) & = x^3 - 4x^2 + x + 1
\end{align*}
are invariant under the action of $x \rightarrow \phi(x) = \frac{1}{1-x}$:
\begin{equation*}
(1-x)^6A(\phi(x)) = A(x), \ \ (1-x)^3B(\phi(x)) = -B(x).
\end{equation*}
Thus, $f_{1728}(x)$ splits into the product of the factors
$$C_{\pm}(x) = A(x) \pm 2\sqrt{7}x(x-1)B(x),$$
where
$$\textrm{disc}_x(A(x)-2\sqrt{7}x(x-1)B(x)) = 2^6 \cdot 3^6 \cdot 7^2 \cdot (21+8\sqrt{7}),$$
and
$$(1-x)^6C_{\pm}(\phi(x)) = C_{\pm}(x).$$
Now $\sqrt{7} \in \mathbb{F}_l$ since $\left(\frac{-7}{l}\right) = -1$, and
$$(21+8\sqrt{7})(21-8\sqrt{7}) = -7.$$
It follows that the discriminant of exactly one of the factors $C_{\pm}(x)$ is a square in $\mathbb{F}_l$.  The Pellet-Stickelberger-Voronoi (PSV) theorem (see \cite[p. 485]{h}, \cite[pp.153-155]{nar}, \cite[Appendix]{brm}) implies that one of the factors, say $C_{-}(x)$, has an even number of irreducible factors in $\mathbb{F}_l$, and the other, $C_{+}(x)$, has an odd number of irreducible factors.  The factors of both polynomials divide $\hat H_{7,l}(x)$ in this case, so by Theorem \ref{thm:1}, they must be linear or quadratic.  If one of the factors $\tilde q(x)$ is quadratic, there must be three quadratic factors in its orbit under the action of $x \rightarrow \phi(x)$ unless $\tilde q(x) = x^2-x+1$, which is fixed by $\phi(x)$.  However,
$$\textrm{Res}_x(x^2-x+1,f_{1728}(x)) = 2^6 \cdot 3^3 \cdot 7,$$
so $x^2-x+1$ does not occur as a factor of $f_{1728}(x)$ over $\mathbb{F}_l$.  It follows that the factors of $C_{-}(x)$ must all be linear, and the factors of $C_{+}(x)$ must be quadratic.  This proves that $f_{1728}(x)$ has exactly $6$ linear factors over $\mathbb{F}_l$.
\end{proof}

\noindent {\bf Remark.} Since $\textrm{Gal}(q(x)/\mathbb{Q}) = D_4$, it is not hard to see that $\textrm{Gal}(f_{1728}(x)/\mathbb{Q})$ is a subgroup of the wreath product $\mathbb{Z}/3\mathbb{Z} \ wr \ D_4$, which is a group of order $3^4 \cdot 8 = 648$.  See, for example, \cite[Thm. B]{mp} with $\sigma(x) = \phi(x)$.  Even though $\phi(x)$ is not a polynomial, the same arguments apply.  In fact, we could take $\sigma(x)$ to be the polynomial quotient on dividing $f_{1728}(x)$ by $x-1$, since the remainder is $1$ and consequently $(x-1)\sigma(x) \equiv -1$ (mod $f_{1728}(x)$), i.e., $\sigma(x) = \phi(x)$ on roots of $f_{1728}(x)$.  Note that $f_{1728}(x)$ is not a normal polynomial by Part (b) of the above proposition (otherwise it would split into irreducibles of the same degree mod $l$), so its Galois group must have order greater than $12$.  In the following proposition we show that the Galois group of $f_{1728}(x)$ has order $24$.  \medskip

\begin{prop} The splitting field $L$ of the polynomial $f_{1728}(x)$ over $\mathbb{Q}$ coincides with $\Sigma_{4\mathfrak{p}_7}$, the ray class field of conductor $4\mathfrak{p}_7$ over $k=\mathbb{Q}(\sqrt{-7})$, where $\mathfrak{p}_7^2 \cong 7$ in $k$.
\label{prop:3}
\end{prop}

\begin{proof}
Note that $f_{1728}(x)$ factors over the field $\mathbb{Q}(r)$: $f_{1728}(x) = N_{\mathbb{Q}(r)/\mathbb{Q}}(g(x))$,
where
\begin{align*}
g(x) &= x^4 + \frac{1}{7}(-10r^2 + 72r - 54)x^3 + \frac{1}{7}(18r^2 - 124r + 93)x^2\\
& \ + \frac{1}{7}(-12r^2 + 78r - 62)x + \frac{1}{7}(2r^2 - 20r + 15).
\end{align*}
By the above proof, it also factors over the field $\mathbb{Q}(\sqrt{7})$.  Hence, the real abelian sextic field $\mathbb{Q}(\sqrt{7},r)$ is a subfield of $L$.   Since
$$\textrm{disc}_x f_{1728}(x) = -2^{36} \cdot 3^{24} \cdot 7^{11},$$
it is also clear that $\sqrt{-7} \in L$.  Thus, $\mathbb{Q}(\zeta_{28}) \subset L$.  Now it can be checked using the cubic resolvent of the polynomial $g(x)$ that $\textrm{Gal}(g(x)/\mathbb{Q}(r)) = D_4$ has order $8$.  On the other hand, the above remark shows that at most $8$ can divide the degree of $L/\mathbb{Q}$.  It follows that the conjugates $g^\phi, g^{\phi^2}$ of $g(x)$, whose groups over $\mathbb{Q}(r)$ are also dihedral of order $8$, must split completely in the splitting field of $g(x)$ over $\mathbb{Q}(r)$.  Hence $[L:\mathbb{Q}(r)] = 8$ and therefore $[L:\mathbb{Q}]=24$.  Thus, $[L:\mathbb{Q}(\sqrt{-7})]=12$ and $[L:\mathbb{Q}(\zeta_{28})]=2$.  If $J = \langle \tau \rangle \subset G = \textrm{Gal}(L/k)$ corresponds to $\mathbb{Q}(\zeta_{28})$ in the Galois correspondence, then $J$ is a normal subgroup of $G$ for which $G/J = \textrm{Gal}(\mathbb{Q}(\zeta_{28})/k) \cong \mathbb{Z}_2 \times \mathbb{Z}_3$ is cyclic, which implies that $G$ is abelian.  \medskip

The discriminant calculation 
$$\textrm{disc}_x(A(x)-2\sqrt{7}x(x-1)B(x)) = 2^6 \cdot 3^6 \cdot 7^2 \cdot (21+8\sqrt{7})$$
in the proof of Prop. \ref{prop:2} implies that $\sqrt{21+8\sqrt{7}} \in L$.  Also, $k(\sqrt{21+8\sqrt{7}})$ is a quartic extension of $k$.  Hence $L = k(\sqrt{21+8\sqrt{7}}, r)$.  Now the discriminant of $\mathbb{Q}(\zeta_7)/\mathbb{Q}$ is $-7^5$, which yields that the norm of the discriminant $\mathfrak{d}_1$ of $\mathbb{Q}(\zeta_7)/k$ is $7^2$.  Thus $\mathfrak{d}_1 \cong 7 \cong \wp_7^2$, so that $\mathfrak{f}_1 \cong \wp_7$ is the conductor of the cubic extension $\mathbb{Q}(\zeta_7)/k$.  Further, the roots of the cubic resolvent of $x^4-42x^2-7$ are $0, -42+2\sqrt{-7}, -42-2\sqrt{-7}$ (see \cite[p. 191]{vdw}), from which we obtain that
$$k(\sqrt{21+8\sqrt{7}}) = k\left(\sqrt{\frac{21+\sqrt{-7}}{2}},\sqrt{\frac{21-\sqrt{-7}}{2}}\right).$$
Note that
$$\left(\sqrt{\frac{21+\sqrt{-7}}{2}} + \sqrt{\frac{21-\sqrt{-7}}{2}}\right)^2 = 21+8\sqrt{7}.$$
Since
$$\frac{21+\sqrt{-7}}{2} = \left(\frac{1+\sqrt{-7}}{2}\right)^4 \sqrt{-7},$$
this gives that
$$k(\sqrt{21+8\sqrt{7}}) = k(\sqrt[4]{-7}, \sqrt{-\sqrt{-7}})=k(\sqrt[4]{-7},i).$$
The conductor $\mathfrak{f}_2$ of the field $k(\sqrt[4]{-7})$ divides $4\wp_7$, while the conductor of $k(i)$ equals $\mathfrak{f}_3 = (4)$.  It follows that the conductor of $L = k(r, \sqrt[4]{-7}, i)$ is $\textrm{lcm}(\mathfrak{f}_1,\mathfrak{f}_2,\mathfrak{f}_3) = \mathfrak{f}_{L/k} = 4\wp_7$.  Since $[\Sigma_{4\wp_7}:k] = 12$, it follows that $L = \Sigma_{4\wp_7}$.
\end{proof}

Turning now to the cases $l \equiv 3, 5$ (mod $7$), we have:

\begin{prop} Assume $l > 3, l \neq 7$ is a prime with $l \equiv 3, 5$ (mod $7$). \smallskip

\noindent (a) If $l \equiv 2$ (mod $3$), the polynomial
$$f_0(x) = (x^2-x+1)(x^6-11x^5+30x^4-15x^3-10x^2+5x+1)$$
has exactly $2$ distinct irreducible cubic factors modulo $l$. \smallskip

\noindent (b) If $l \equiv 3$ (mod $4$), the polynomial
\begin{align*} 
f_{1728}(x) =& x^{12}-18x^{11}+117x^{10}-354x^9+570x^8-486x^7+273x^6\\
\notag & -222x^5+174x^4-46x^3-15x^2+6x+1
\end{align*}
has exactly $2$ distinct irreducible cubic factors modulo $l$.
\label{prop:4}
\end{prop}

\begin{proof} For (a), this is immediate from the proof of Proposition \ref{prop:2}.  The conditions $l \equiv 2$ mod $3$ and $l \equiv 3, 5$ mod $7$ give that $l \equiv 17$, resp. $5$ mod $21$.  Since
$$l^3 \equiv -4^3 \ \textrm{or} \ 5^3 \equiv -1 \ (\textrm{mod} \ 21),$$
$l$ splits into primes of degree $3$ in the field $F^+$.  This yields the assertion. \medskip

For (b), we have $f_{1728}(x) = C_{+}(x) C_{-}(x)$ as in Proposition \ref{prop:2}.  By Theorem \ref{thm:1}, the factors of $f_{1728}(x)$ over $\mathbb{F}_l$ can only be $x^2-x+1$, cubic or sextic.  We have seen that the first possibility does not occur.  As before, if $C_{-}(x)$ has an even number of irreducible factors mod $l$, then $C_{+}(x)$ has an odd number of irreducible factors.  It follows that $C_{-}(x)$ is a product of two cubics and $C_{+}(x)$ is an irreducible sextic.  This proves the assertion.
\end{proof}

\begin{lem}
\begin{enumerate}[(a)]

\item Treating $z, j$ as indeterminates, the Galois group $G$ of the normal closure $N$ of the function field extension $\mathbb{Q}(z)/\mathbb{Q}(j)$ defined by
$$F(z,j) = (z^2-3z+9)(z^2-11z+25)^3-j(z-8) = 0$$
is
$$G = \textrm{Gal}(N/\mathbb{Q}(j)) = \textsf{PGL}(2, 7).$$

\item If $l \neq 7$ is a prime, the Galois group of the normal closure $\overline{N}$ of the extension $\mathbb{F}_l(z)/\mathbb{F}_l(j)$ is a subgroup $\overline{G}$ of $G$, considered as a permutation group on the roots of $F(z,j)$ (mod $l$).
\end{enumerate}
\label{lem:3}
\end{lem}

\begin{proof} (a) First note the following factorizations into irreducibles:
\begin{align*}
F(z,-15^3)  = & (z - 1)(z^7 - 35z^6 + 511z^5 - 4081z^4 + 19754z^3 - 60550z^2 \\
& \ + 115500z - 113625),\\
F(z,-96^3)  = & (z^2 - 11z + 73)(z^6 - 25z^5 + 198z^4 - 589z^3 + 2902z^2\\
& \  - 5385z - 95031),\\
F(z,66^3) = & z^8 - 36z^7 + 546z^6 - 4592z^5 + 23835z^4 - 80304z^3 + 176050z^2\\
& \  - 519996z + 2440593.
\end{align*}
In these cases the factor of larger degree has Galois group $D_7, D_6, D_8$, respectively; and each is a subgroup of $G$.  (The $7$-th degree factor belongs to the ring class field of conductor $f=7$ over $\mathbb{Q}(\sqrt{-7})$; the $6$-th degree factor belongs to the ring class field of conductor $f = 7$ over $\mathbb{Q}(\sqrt{-19})$; the $8$-th degree factor belongs to the ring class field of conductor $f=2\cdot 7$ over $\mathbb{Q}(i)$.)  Hence, $|G| \ge lcm(12,14,16) = 336$.  Since $\textsf{PGL}(2,7)$ and $S_8$ are the only primitive permutation groups of degree 8 containing odd permutations, we just have to show that $|G| < |S_8|$.  (See Table 1 in \cite{sim}.)  Using the field $k=\mathbb{Q}(\zeta_7)$, we consider the chain of fields
$$\mathbb{Q}(j) \subset k(j) \subset k(z) \subset k(h) \subset k(s,h),$$
where the relationships between the indeterminates $z, h, s$ are defined by
$$h^3 - 3h + 1 - z h(h-1) = 0, \ \ s^7 = h(h-1)^2.$$
Then by our previous results the last four fields in the above chain are isomorphic to the modular function fields, respectively, in the chain
$$\textsf{K}_\Gamma \subset \textsf{K}_{\Gamma_0(7)} \subset \textsf{K}_{\Gamma_1(7)} \subset \textsf{K}_{\Gamma(7)},$$
where $\textsf{K}_{\Gamma(7)}$ is normal over $\textsf{K}_\Gamma$ with Galois group $\Gamma/\Gamma(7) \cong \textsf{PSL}(2,7)$.  It follows that as permutation groups on the roots of $F(z,j) = 0$,
$$\textrm{Gal}(k(s,h)/k(j)) \cong \textsf{PSL}(2,7).$$
Furthermore, $k(s,h)$ is normal over $\mathbb{Q}(j)$, since any isomorphism of $k(s,h)$ (into an algebraic closure) which fixes $\mathbb{Q}(j)$ also sends the field $k(j)$ to itself ($k/\mathbb{Q}$ being abelian) and $k(z)$ to a conjugate field $k(z')$ inside $k(s,h)$ and therefore -- since $k(s,h)$ is the normal closure of $k(z')/k(j)$ -- maps $k(s,h)$ to itself.  Thus, the normal closure $N$ of $\mathbb{Q}(z)/\mathbb{Q}(j)$ is a subfield of $k(s,h)$, which gives that
$$[N:\mathbb{Q}(j)] \le [k(s,h): k(j)] [k:\mathbb{Q}] = 6 |\textsf{PSL}(2,7)| = 1008 < 8!.$$
This proves that
$$\textrm{Gal}(N/\mathbb{Q}(j)) = \textsf{PGL}(2,7).$$

(b) This is immediate from van der Waerden's theorem \cite[p. 198]{vdw}, using
$$\mathfrak{R} = \mathbb{Z}[j], \ \mathfrak{p} = (l) = l \mathbb{Z}[j], \ \overline{\mathfrak{R}} = \mathfrak{R}/\mathfrak{p} = \mathbb{F}_l[j],$$
and the fact that
$$\textrm{disc}_z(F(z,j)) = -7^7j^4(j-1728)^4.$$
This discriminant calculation implies that $F(z,j)$ and $\overline{F}(z,j) = F(z,j)$ mod $\mathfrak{p}$ are separable.  Note also that $\overline{F}(z,j)$ is absolutely irreducible over $\mathbb{F}_l$, for $l \neq 7$, since
$$\textrm{Res}_z((z^2-3z+9)(z^2-11z+25),z-8) = 7^2.$$
\end{proof}

\noindent {\bf Remark.} The proof of Lemma \ref{lem:3}  yields a polynomial with Galois group $\textsf{PSL}(2,7)$ over $k(j)$, which can be used to give a polynomial with Galois group $\textsf{PSL}(2,7)$ over $\mathbb{F}_l(j)$, for primes $l$ for which $\left(\frac{-7}{l}\right) = +1$.  See \cite{efm}. \medskip

We can now draw the following conclusion.  The Galois group $\overline{G}$ of the normal closure $\overline{N}$ of $\mathbb{F}_l(z)/\mathbb{F}_l(j)$ is a subgroup of $\textsf{PGL}(2,7)$, so that the action of $\overline{G}$ on the roots of the $8$-th degree polynomial
$$F(z,j) = (z^2-3z+9)(z^2-11z+25)^3-j(z-8)$$
is isomorphic to its action as a group of linear fractional transformations on the projective line $\textsf{P}(1,\mathbb{F}_7)$.  A non-identity linear fractional map has no more than two fixed points.  Now we use the correspondence between the cycle type of a conjugacy class of permutations $\big(\frac{\overline{N}/\mathbb{F}_l(j)}{\mathfrak{p}_a}\big)$ in $\overline{G} = \textrm{Gal}(\overline{N}/\mathbb{F}_l(j))$ and the factorization of the polynomial $F(z,j) 
\equiv F(z,a)$ modulo $\mathfrak{p}_a$, for the prime divisor $\mathfrak{p}_a$ of the function field $\mathbb{F}_l(j)$ which is the numerator divisor of $j-a$, for a given $a \in \mathbb{F}_l$.  Note that
$$\left(\frac{\overline{N}/\mathbb{F}_l(j)}{\mathfrak{p}_a}\right) = \{1\} \subset \overline{G}$$
if and only if $F(z,a)$ splits into linear factors modulo $l$.  By the theorem of Pellet-Stickelberger-Voronoi, 
$$\left(\frac{\textrm{disc}_z(F(z,a))}{l}\right) = \left(\frac{-7}{l}\right) = (-1)^{r_l}, \ \ l \neq 7, \ a \not \equiv 0, 1728 \ (\textrm{mod} \ l);$$
where $r_l$ is the number of distinct irreducible factors of $F(z,a)$ (mod $l$).  It follows that if $l \equiv 3, 5, 6$ (mod $7$), the $8$-th degree polynomial $F(z,a)$ has an odd number of irreducible factors (mod $l$), whence it cannot split completely and
$$\left(\frac{\overline{N}/\mathbb{F}_l(j)}{\mathfrak{p}_a}\right) \neq \{1\},$$
for these primes.  Hence, when $l \equiv 3, 5, 6$ (mod $7$), then $F(z,a)$ can have no more than $2$ linear factors mod $l$, when $j \equiv a \neq 0, 1728$ is a supersingular invariant in $\mathbb{F}_l$.  We would like to show that $F(z,a)$ always has exactly two linear factors in this situation. \medskip

Now we know that for $d = l$ or $4l$ ($l>3$), the class equation $H_{-d}(X)$ splits into linear factors (mod $l$).  In fact,
\begin{align*}
H_{-l}(X) & \equiv (X-1728)R(X)^2,\\
H_{-4l}(X) & \equiv (X-1728)S(X)^2 \ (\textrm{mod} \ l), \ l \equiv 3 \ (\textrm{mod} \ 4),
\end{align*}
and
$$H_{-4l} \equiv T(X)^2 \ (\textrm{mod} \ l), \ l \equiv 1 \ (\textrm{mod} \ 4);$$
where $R(X),S(X), T(X)$ are products of distinct linear factors and the gcd $(R(X),S(X)) = 1$ in $\mathbb{F}_l[X]$.  The roots of these polynomials in $\mathbb{F}_l$ are exactly the supersingular $j$-invariants in characteristic $l$ which lie in the prime field.  (See \cite[Props. 9, 11, pp. 95, 104]{brm} and the references cited in that paper.)  \smallskip

This discussion leads to a proof of the following theorem.

\begin{thm} (a) If $l \equiv 6$ (mod $7$) and $j \not \equiv 0, 1728$ (mod $l$) is a supersingular $j$-invariant in $\mathbb{F}_l$, then the polynomial
\begin{align}
\label{eqn:12} G(x,j) = & (x^2 - x + 1)^3 (x^6 + 229x^5 + 270x^4 - 1695x^3 + 1430x^2 - 235x + 1)^3\\
\notag &\ - jx(x-1)(x^3 - 8x^2 + 5x + 1)^7.
\end{align}
has exactly $6$ distinct linear factors over $\mathbb{F}_l$. \smallskip

(b) If $l \equiv 3, 5$ (mod $7$), and $j \not \equiv 0, 1728$ (mod $l$) is a supersingular $j$-invariant in $\mathbb{F}_l$, then the polynomial $G(x,j)$ has exactly two irreducible cubic factors over $\mathbb{F}_l$.
\label{thm:7}
\end{thm}

\begin{proof}
Let $0, 1728 \neq j \in \mathbb{F}_l$ be a supersingular $j$-invariant and let $d = l$ or $4l$ be chosen so that $H_{-d}(j) \equiv 0$ (mod $l$).  By the above congruences, there is a prime divisor $\mathfrak{p}$ of $\mathfrak{l}$ in $\Omega_f$ and a root $\mathfrak{j}$ of $H_{-d}(X)$ for which
$$j \equiv \mathfrak{j} \ (\textrm{mod} \ \mathfrak{p}).$$
By (\ref{eqn:7}), (\ref{eqn:5}) and (\ref{eqn:4.5}),
\begin{align*}
j(w/7) & = \frac{(z^2-3z+9)(z^2+229z+505)^3}{(z-8)^7}, \ \ z = z(w/7),\\
& = \frac{(h^2 - h + 1)^3(h^6 + 229h^5 + 270h^4 - 1695h^3 + 1430h^2 - 235h + 1)^3}{h(h - 1)(h^3 - 8h^2 + 5h + 1)^7}\\
& = j_{7,7}(h(w/7)),
\end{align*}
with $h = h(w/7)$, and we have $G(h(w/7),j(w/7)) = 0$.  Since $\mathfrak{j}$ is a conjugate of $j(w/7)$ over $K$, it follows that some conjugate $\eta^\sigma$ of $\eta = h(w/7)$ over $K$ is a root of $G(x, \mathfrak{j}) = 0$.  Hence, the minimal polynomial $\mu_1(x)$ of $\eta^\sigma$ over $\Omega_f$ divides $G(x, \mathfrak{j})$.  Using (\ref{eqn:6}), (\ref{eqn:1}), (\ref{eqn:4.5}) and (\ref{eqn:9.5}), we have that
\begin{align*}
j(w) & = \frac{(z^2-3z+9)(z^2-11z+25)^3}{(z-8)}, \ \ z = z(w/7),\\
& = \frac{(\eta^2 - \eta + 1)^3(\eta^6 -11\eta^5 + 30\eta^4 - 15\eta^3 -10\eta^2 +5\eta + 1)^3}{\eta^7(\eta - 1)^7(\eta^3 - 8\eta^2 + 5\eta + 1)}\\
& = j_7(\eta) = j_7(h(w/7)) = j_{7,7}(T_1(h(w/7)) = j_{7,7}(h(-1/w)) = j_{7,7}(\xi).
\end{align*}
This implies that the minimal polynomial $\mu_2(x)$ over $\Omega_f$ of some conjugate $\xi^{\sigma'}$ of $\xi = h(-1/w)$ divides $G(x, \mathfrak{j})$, as well.  Since $\Sigma_{\wp_7'} \Omega_f$ and $\Sigma_{\wp_7} \Omega_f$ are linearly disjoint over $\Omega_f$, Theorems 5 and 6 imply that $\mu_1(x)$ and $\mu_2(x)$ are relatively prime.  Thus, we have that
$$\mu_1(x) \mu_2(x) \mid G(x,\mathfrak{j}),$$
giving that
$$\mu_1(x) \mu_2(x) \mid G(x,j) \ \textrm{mod} \ \mathfrak{p}.$$
Now assume $l \equiv 6$ (mod $7$).  Then $\mathfrak{p} \mid \mathfrak{l}$, so that $\mathfrak{p}$ splits in both fields $\Sigma_{\wp_7'} \Omega_f$ and $\Sigma_{\wp_7} \Omega_f$.  Further, since $j \neq 0, 1728$ (mod $\mathfrak{p}$), the polynomial $G(x,j) \in \mathbb{F}_l[x]$ has distinct roots mod $\mathfrak{p}$.  Then both $\mu_i(x)$ split into linear factors mod $\mathfrak{p}$, so that $G(x,j)$ has at least $6$ linear factors modulo $\mathfrak{p}$, and therefore also modulo $l$. \medskip

If $x \in \mathbb{F}_l$ is any root of $G(x,j)$, then $z = \frac{x^3-3x+1}{x(x-1)} \in \mathbb{F}_l$ is a root of the minimal polynomial $m_d(X)$ of $z(w/7)$ over $\mathbb{Q}$, for which
$$j \equiv \frac{(z^2-3z+9)(z^2+229z+505)^3}{(z-8)^7} \ (\textrm{mod} \ l).$$
(Note that $G(0,j) = G(1,j) = 1$, so $x \not \equiv 0, 1$ (mod $l$).  Further, $z \not \equiv 8$ (mod $l$) since $x^3-8x+5x+1 \not \equiv 0$ (mod $l$), which follows from the fact that the resultant of $(x^2 - x + 1)(x^6 + 229x^5 + 270x^4 - 1695x^3 + 1430x^2 - 235x + 1)$ and $x^3-8x+5x+1$ equals $7^{14}$.)
Hence, $z$ is a root of the congruence
$$G_1(z,j) = (z^2-3z+9)(z^2+229z+505)^3-j(z-8)^7 \ (\textrm{mod} \ l).$$
Further,
$$(z-8)^8 G_1\left(\frac{8z-15}{z-8},j \right) = 7^{14} \{(z^2-3z+9)(z^2-11z+25)^3 - j(z-8)\} = 7^{14} F(z,j).$$
Different orbits of roots in $\mathbb{F}_l$ of $G(x,j)$ with respect to the group $\langle \phi \rangle$ correspond to different roots $z \in \mathbb{F}_l$ of $G_1(z,j) \equiv 0$ (mod $l$). The discussion following Lemma \ref{lem:3} shows that $F(z,j) \equiv 0$ (mod $l$) can have no more than two roots in $\mathbb{F}_l$, so the same is true for $G_1(z,j) \equiv 0$ and this implies that $G(x,j)$ can have no more than $6$ linear factors modulo $l$.  This proves part (a). \medskip

To prove part (b), we just have to note that $G(x,j)$ can have no linear factors (mod $l$) when $l \equiv 3, 5$ (mod $7$), by Theorem \ref{thm:1}.  Hence, the factors $\mu_1(x), \mu_2(x)$ in the above proof must remain irreducible modulo $l$, showing that $G(x,j)$ has at least two distinct irreducible cubic factors modulo $l$.  Furthermore, any irreducible cubic factor of $\hat H_{7,l}(x)$ -- and hence of $G(x,j)$ -- over $\mathbb{F}_l$ has the form
$$x^3 +ax^2 - (a+3)x+1 = x^3-3x+1+a(x^2-x),$$
by \cite[Thm. 5.3]{mor1}.  Hence, $z = \frac{x^3-3x+1}{x(x-1)} \equiv -a \in \mathbb{F}_l$, from which we conclude that $G(x,j)$ can have no more than two irreducible cubic factors (mod $l$).  This completes the proof.
\end{proof}

Combined with Propositions 2 and 4, this theorem yields a proof of Conjecture \ref{conj:1}. \medskip

\noindent {\it Proof of Conjecture \ref{conj:1}.} By Proposition \ref{prop:2}, when $l \equiv 6$ (mod $7$), each of the terms
$$f_0(x) = (x^2-x+1)(x^6-11x^5+30x^4-15x^3-10x^2+5x+1)$$
and
\begin{align*}
f_{1728}(x) = & x^{12}-18x^{11}+117x^{10}-354x^9+570x^8-486x^7+273x^6\\
& \ -222x^5+174x^4-46x^3-15x^2+6x+1
\end{align*}
in the formula for $\hat H_{7,l}(x)$ has, when supersingular, exactly $6$ linear factors mod $l$.  The same holds for $G(x,j)$ by Theorem \ref{thm:7}(a), and therefore also for the factors
\begin{align*}
H(x,j) =& \  (x^2-x+1)^3(x^6-11x^5+30x^4-15x^3-10x^2+5x+1)^3\\
& \ -j x^7(x-1)^7(x^3-8x^2+5x+1)
\end{align*}
of $x^{7n_l}(x-1)^{7n_l}(x^3-8x^2+5x+1)^{n_l}J_l(j_7(x))$.  This is because
\begin{equation*}
((1-r)x-1)^{24}G(T_1(x),j) = 7^{14} \varepsilon H(x,j),
\end{equation*}
where $\varepsilon = 413283046371r^2 - 290993856257r - 56645954512 = r^8(r-1)^8$ is a unit in $\mathbb{Q}(r)$.   Note that  $x^3-8x^2+5x+1$ splits mod $l$ in this case, so linear factors of $G(x,j)$ are in $1-1$ correspondence with linear factors of $H(x,j)$.  Since we have $6$ linear factors of $\hat H_{7,l}(x)$ (mod $l$) for each supersingular $j$-invariant in characteristic $l$, this proves Conjecture \ref{conj:1}A. \medskip

To prove Conjecture \ref{conj:1}B, we note that if $g(x) = x^3+ax^2-(a+3)x+1$, then
$$((1-r)x-1)^3 g(T_1(x)) = r(1-r)(a+8)\left\{x^3-\left(\frac{8a+15}{a+8}\right)x^2+\left(\frac{5a - 9}{a + 8}\right) x+1\right\},$$
where
$$\frac{5a - 9}{a + 8} = -\left(-\frac{8a+15}{a+8}\right)-3.$$
Since $a \not \equiv -8$ (mod $l$), the irreducible cubic factors of $G(x,j)$ over $\mathbb{F}_l$ are in $1-1$ correspondence with the irreducible cubic factors of $H(x,j)$.  Now part Conjecture \ref{conj:1}B follows from Proposition \ref{prop:4} and Theorem \ref{thm:7}(b).
$\square$ \bigskip

\section{A modular factorization.}

As in \cite[Thm. 4.1]{mor1}, Theorem \ref{thm:7} allows us to prove the following.

\begin{thm} Assume $l \neq 7$ is a prime satisfying $l > 3$. \smallskip

(a) If $l \equiv 6$ (mod $7$), the minimal polynomial $P_d(x)$ of $h(w/7)$ over $\mathbb{Q}$ satisfies
$$P_d(x) \equiv \prod_{i=1}^{3h}{(x-a_i)^2} \ (\textrm{mod} \ l),$$
where the $a_i$ are distinct (mod $l$) and $h=\textsf{h}(-d)$, $d = l$ or $4l$, is the class number of the order $\textsf{R}_{-d}$ of discriminant $-d$ in $K = \mathbb{Q}(\sqrt{-l})$. \smallskip

\noindent (b) If $l \equiv 3,5$ (mod $7$), $P_d(x)$ factors into a product of the squares of $\textsf{h}(-d)$ distinct irreducible cubics (mod $l$).
\label{thm:9}
\end{thm}

\begin{proof}
(a) The same argument used in the proof of \cite[Thm. 3.1]{mor1} shows that $P_d(x)$ factors into the squares of $3h = 3h(-d)$ linear factors (mod $l$).  To show that these $3h$ linear factors are distinct, consider the minimal polynomials $\mu_1(x), \mu_2(x)$ over $\Omega_f$ of $\eta^\sigma$ and $\xi^{\sigma'}$ occurring in the proof of Theorem \ref{thm:7}.  These polynomials divide $P_d(x)$ (the minimal polynomial of $\eta$ and $\xi$ over $\mathbb{Q}$) and are relatively prime (mod $\mathfrak{p}$), where $\mathfrak{p}$ is any prime divisor of $\mathfrak{l}$ in $\Omega_f$.  If $l \equiv 6$ (mod $7$), the proof of Theorem \ref{thm:7}(a) shows that $P_d(x)$ has $6$ distinct linear factors modulo $\mathfrak{p}$ and therefore modulo $l$, for each supersingular $j$-invariant $j \not \equiv 0, 1728$ (mod $l$).  Clearly, linear factors corresponding to distinct values of $j$ are distinct from each other (mod $l$). \medskip

Furthermore, when $j=0$ is supersingular in characteristic $l$, at least one of the roots of the polynomial
$x^6-11x^5+30x^4-15x^3-10x^2+5x+1$ (mod $l$) is a root of $P_d(x)$ (mod $l$), by the same argument as in the proof of Theorem \ref{thm:7}(a).  (Take $\mathfrak{j}$ to be a root of $H_{-4l}(X)$ which is divisible by a prime divisor $\mathfrak{p}$ of $\mathfrak{l}$.  Note that $j=0$ is never a root of $H_{-l}(X)$ when $l > 3$ and $l \equiv 3$ (mod $4$).  See \cite[p. 236]{mor1}.)  The same holds for $j=1728$ and one of the roots (mod $l$) of the polynomial $f_{1728}(x)$ from Proposition \ref{prop:2}(b).
\medskip

Now we use the fact that the set of roots of $P_d(x)$ is invariant under the action of the group
$$\mathcal{G}_7 = \langle \phi, T_1 \rangle = \langle T_1, T_2, T_3 \rangle = \{1, \phi, \phi^2, T_1, T_2, T_3\},$$
by Theorem \ref{thm:4}.  Note that the fixed points of $\phi(x) = \frac{1}{1-x}$ and $\phi^2(x)$ are the roots of $x^2-x+1 = 0$, and the fixed points of $T_1(x) = \frac{x-r}{(1-r)x-1}$ are
$$x = \frac{1}{7}r^2 - \frac{10}{7}r + \frac{11}{7} = \eta, \ \ -\frac{15}{7}r^2 + \frac{108}{7}r + \frac{17}{7} = -6\eta^2 + 3\eta + 14,$$
which are roots of $x^3 - x^2 - 2x + 1$ and $x^3 - 15x^2 + 12x + 1$, respectively.  (Note also that $j_7(\eta) = -3375 = -15^3$ is the $j$-invariant for the discriminant $D = -7$ and $j_7(-6\eta^2 + 3\eta + 14) = 255^3$ is the $j$-invariant for the discriminant $D = -28$; see equation (\ref{eqn:1}).)  Furthermore,
\begin{align*}
& \textrm{Res}_x(x^2 - x + 1, x^3 - x^2 - 2x + 1) = 7,\\
& \textrm{Res}_x(x^2 - x + 1, x^3 - 15x^2 + 12x + 1) = 3^3 \cdot 7,
\end{align*}
so that no element of $\mathbb{F}_l$ is fixed by all the elements of $\mathcal{G}_7$.  In addition, roots of $x^2-x+1$ correspond to $j = 0$ and discriminant $D = -3$, where $j=0$ is supersingular if and only if $\left(\frac{-3}{l}\right) = -1$.  Hence, the roots of $x^2-x+1$ never lie in $\mathbb{F}_l$ (or in $\mathbb{F}_{l^3}$) when $j=0$ is supersingular.  Certainly, no two of the mappings $T_i$ have the same fixed points, since the fixed points for different $T_i$ are conjugate over $\mathbb{Q}$ and
$$\textrm{disc}(x^3 - x^2 - 2x + 1) = 7^2, \ \ \textrm{disc}(x^3 - 15x^2 +12x + 1) = 3^6 \cdot 7^2,$$
neither of which is divisible by $l$.  (Alternatively, any two of the maps $T_i, T_j$ generate $\mathcal{G}_7$.)  It follows that the stabilizer in $\mathcal{G}_7$ of a root of $P_d(x)$ in $\mathbb{F}_l$ is either $\{1\}$ or $\{1, T_i\}$ for some $i$, hence the orbit of such a root consists of either $6$ or $3$ elements.  We also have
\begin{align*}
\textrm{Res}_x(x^6 - 11x^5 + 30x^4 - 15x^3 - 10x^2 + 5x + 1, &x^3 - x^2 - 2x + 1) = 3^3 \cdot 5^3,\\
\textrm{Res}_x(x^6 - 11x^5 + 30x^4 - 15x^3 - 10x^2 + 5x + 1, &x^3 - 15x^2 +12x + 1)\\
& \ \ \ = 5^3 \cdot 17^3,
\end{align*}
so that roots of $f_0(x) \equiv 0$ (mod $l$) in (\ref{eqn:10}) are not fixed by any of the maps $T_i$.  Thus, when $l \equiv 2$ (mod $3$) and $6$ (mod $7$), there are $6$ roots of $P_d(x) \equiv 0$ (mod $l$) corresponding to $j=0$.  \medskip

Now we can repeat the argument in the proof of \cite[Thm. 4.1]{mor1} practically verbatim.  By the congruences for $H_{-d}(X)$ given before Theorem \ref{thm:7}, there are either $\frac{\textsf{h}(-l)-1}{2}, \frac{\textsf{h}(-4l)-1}{2},$ or $\frac{\textsf{h}(-4l)}{2}$ distinct supersingular $j$-invariants, different from $j=1728$, which are roots of $H_{-d}(x) \equiv 0$ (mod $l$), depending on whether $d=l, 4l$ (with $l \equiv 3$ mod $4$) or $d=4l$ and $l \equiv 1$ modulo $4$.  Thus, $P_d(x) \equiv 0$ has at least
\begin{align*}
6\frac{\textsf{h}(-l)-1}{2} &= 3\textsf{h}(-l)-3, \ (\textrm{when} \ d=l),\\
6\frac{\textsf{h}(-4l)-1}{2} &= 3\textsf{h}(-4l)-3, \ (\textrm{when} \ d=4l, \ l \equiv 3 \ \textrm{mod} \ 4), \ \textrm{or}\\
6\frac{\textsf{h}(-4l)}{2} &=3\textsf{h}(-4l), \ (\textrm{when} \ d=4l, \ l \equiv 1 \ \textrm{mod} \ 4),
\end{align*}
distinct roots corresponding to supersingular $j$-invariants $j \not \equiv 1728$ (mod $l$).  There are at least three roots of $P_d(x) \equiv 0$ in $\mathbb{F}_l$ corresponding to $j=1728$, when $l \equiv 3$ (mod $4$) (since the orbit under $\mathcal{G}_7$ of a root of $P_d(x)$ in $\mathbb{F}_l$ corresponding to $j=1728$ has at least $3$ elements).  On the other hand, there cannot be more than three roots each for the polynomials $P_l(x)$ and $P_{4l}(x)$, because $P_d(x)$ has no more than $3\textsf{h}(-d)$ distinct roots (mod $l$).  Thus, both polynomials have exactly $3$ roots corresponding to $j=1728$, and this proves that $P_d(x) \equiv 0$ has $3\textsf{h}(-d)$ distinct roots in $\mathbb{F}_l$ in all cases.  \medskip

(b) For the primes $l = 5, 17$ we have $d = 20, 68$, and
\begin{align*}
P_{20}(x) & = x^{12} - 2x^{11} - 19x^{10} - 10x^9 + 530x^8 - 1430x^7 + 1377x^6\\
& \ \ - 302x^5 - 250x^4 + 90x^3 + 25x^2 - 10x + 1\\
& \equiv (x^3 + 2x + 1)^2 (x^3 + 4x^2 + 3x + 1)^2 \ (\textrm{mod} \ 5)
\end{align*}
and
\begin{align*}
P_{68}(x) & = x^{24} - 12x^{23} + 1550x^{22} - 36832x^{21} + 391173x^{20} - 2647784x^{19}\\
& \ \ + 14397706x^{18} - 66497132x^{17} + 241323226x^{16} - 626494580x^{15}\\
& \ \ + 1111176102x^{14} - 1314407496x^{13} + 999381181x^{12} - 446384584x^{11}\\
& \ \ + 89881766x^{10} - 2645876x^9 + 7200858x^8 - 5340652x^7 + 495178x^6\\
& \ \ + 235800x^5 - 34875x^4 + 3744x^3 + 1550x^2 - 12x + 1\\
& \equiv (x^3 + 4x^2 + 10x + 1)^2 (x^3 + 16x^2 + 15x + 1)^2 (x^3 + 6x^2 + 8x + 1)^2\\
& \ \ \times (x^3 + 2x^2 + 12x + 1)^2 \ (\textrm{mod} \ 17).
\end{align*}
Note that each of the cubics in these two congruences has the form $x^3 + ax^2 - (a+3)x + 1$, in agreement with \cite[Thm. 5.3]{mor1}.  Note also that $f(x) = x^3 + ax^2 - (a+3)x + 1$ satisfies
$$(x-1)^3 f(\phi(x)) = (x-1)^3 f\left(\frac{1}{1-x}\right) = f(x),$$
so irreducible cubic factors have orbits of size $1$ or $2$ under the action of $\mathcal{G}_7$. \smallskip

For primes $l \neq 5, 17$, the same arguments as in the proof of (a) show that $P_d(x)$ is divisible by two distinct irreducible  cubics over $\mathbb{F}_l$, for each supersingular $j \not \equiv 1728$ (mod $l$).  Hence, the number of irreducible cubics dividing $P_d(x)$ (mod $l$) is at least
\begin{align*}
2\frac{\textsf{h}(-l)-1}{2} &= \textsf{h}(-l)-1, \ (\textrm{when} \ d=l),\\
2\frac{\textsf{h}(-4l)-1}{2} &= \textsf{h}(-4l)-1, \ (\textrm{when} \ d=4l, \ l \equiv 3 \ \textrm{mod} \ 4), \ \textrm{or}\\
2\frac{\textsf{h}(-4l)}{2} &=\textsf{h}(-4l), \ (\textrm{when} \ d=4l, \ l \equiv 1 \ \textrm{mod} \ 4).
\end{align*}
But there is at least one irreducible cubic corresponding to $j=1728$ which divides $P_d(x)$ for $l \equiv 3$ (mod $4$) and $d=l$ or $d=4l$.  The above formulas show that there can only be one such factor for each $d$.  This proves (b).
\end{proof}

\newtheorem{cor}{Corollary}

\begin{cor} If $\left(\frac{-l}{7}\right) = +1$, and $q_d(x)$ is the minimal polynomial of $\eta = h(w/7)$ over $K = \mathbb{Q}(\sqrt{-l})$, where $w = \frac{v+\sqrt{-d}}{2} \equiv 0$ (mod $\wp_7^2$), as in Theorem \ref{thm:5}, then the prime divisor $\mathfrak{l}$ of $K$ does not divide $\textrm{disc}(q_d(x))$, so the powers of $\eta$ form an $\mathfrak{l}$-integral basis of $\Sigma_{\wp_7'} \Omega_f$ over $K$.
\label{cor:1}
\end{cor}

\begin{proof} This follows in the same way as \cite[Thm 4.2]{mor1}.  Namely,
$$P_d(x) = q_d(x) \bar q_d(x) \equiv q_d(x)^2 \ (\textrm{mod} \ \mathfrak{l}),$$
and the theorem implies that $q_d(x)$ has distinct roots modulo $\mathfrak{l}$.  This proves the assertions.
\end{proof}

In addition to the examples $P_{20}(x)$ and $P_{68}(x)$ mentioned in the proof of Theorem \ref{thm:9}, we note the examples
\begin{align*}
P_{52}(x) & = x^{12} - 58x^{11} + 1645x^{10} - 20442x^9 + 112498x^8 - 250342x^7 + 247929x^6\\
& \ \ - 92782x^5 - 6530x^4 + 6962x^3 + 1073x^2 + 46x + 1\\
& \equiv (x + 2)^2 (x + 4)^2 (x + 5)^2 (x + 7)^2 (x + 8)^2 (x + 10)^2 \ (\textrm{mod} \ 13),
\end{align*}
where $\textsf{h}(-4 \cdot 13) = 2$; and
\begin{align*}
P_{83}(x) & = x^{18} - 137x^{17} + 6765x^{16} - 116316x^{15} + 953694x^{14} - 4517362x^{13}\\
& \ \  + 14472274x^{12} - 31178560x^{11} + 43709339x^{10} - 38751299x^9\\
& \ \  + 20853979x^8 - 6393600x^7 + 1092050x^6 - 92082x^5 - 69346x^4\\
& \ \ + 25892x^3 + 4589x^2 + 119x + 1\\
& \equiv (x + 18)^2 (x + 22)^2 (x + 27)^2 (x + 35)^2 (x + 42)^2 (x + 48)^2 (x + 53)^2\\
& \ \ \times (x + 63)^2 (x + 80)^2 \ (\textrm{mod} \ 83),
\end{align*}
where $\textsf{h}(-83) = 3$.

\section{Appendix.}

Let
$$\eta(\tau)=q^{1/24} \prod_{n \ge 1}{(1-q^n)}, \ \ q = e^{2\pi i \tau},$$
be the Dedekind $\eta$-function.  Also, let
$$h(\tau) = q^{-1} \prod_{n \ge 1}{\frac{(1-q^{7n-3})(1-q^{7n-4})(1-q^{7n-2})^2(1-q^{7n-5})^2}{(1-q^{7n-1})^3(1-q^{7n-6})^3}},$$
a modular function for $\Gamma_1(7)$.  See \cite[p. 156]{du}.

\begin{thm} For $\tau$ in the upper half-plane,
\begin{equation}
\label{eqn:13} \left(\frac{\eta(\tau)}{\eta(7\tau)}\right)^4=\frac{h^3-8h^2+5h+1}{h(h-1)}, \ \ h=h(\tau).
\end{equation}
Also, if $d(\tau)=h(\frac{-1}{7\tau})$, then
\begin{equation}
\label{eqn:14} 49\left(\frac{\eta(7\tau)}{\eta(\tau)}\right)^4=\frac{d^3-8d^2+5d+1}{d(d-1)}, \ \ d=d(\tau).
\end{equation}
\label{thm:10}
\end{thm}

It is clear that (\ref{eqn:14}) follows from (\ref{eqn:13}), using the transformation formula
$$\eta\left(\frac{-1}{\tau}\right) = \sqrt{\frac{\tau}{i}} \eta(\tau).$$

\begin{thm} The function $h(\tau)$ above satisfies the transformation formulas:
\begin{align}
\label{eqn:15} h\left(\frac{3\tau - 1}{7\tau - 2}\right)=& \ \frac{1}{1-h(\tau)},\\
\label{eqn:16} h\left(\frac{2\tau-1}{7\tau-3}\right)= & \ \frac{h(\tau)-1}{h(\tau)}.
\end{align}
\label{thm:11}
\end{thm}

Note that the map $A(\tau)=\frac{2\tau-1}{7\tau-3}$ has order $3$ and $A^2(\tau)= \frac{3\tau - 1}{7\tau - 2}$.  We first prove (\ref{eqn:16}). We use the notation $e(x) = exp(2 \pi i x)$.

\begin{proof} Start with the formulas from \cite[p. 157]{du}:
\begin{align}
\label{eqn:17} s(\tau) = & \ q^{-3/7} \prod_{n \ge 1}{\frac{(1-q^{7n-3})(1-q^{7n-4})}{(1-q^{7n-1})(1-q^{7n-6})}} = e\left(\frac{1}{7}\right) \frac{\theta {1/7 \atopwithdelims [] 1}(7\tau)}{\theta {5/7 \atopwithdelims [] 1}(7\tau)},\\
\label{eqn:18} t(\tau) = & \ q^{-2/7} \prod_{n \ge 1}{\frac{(1-q^{7n-2})(1-q^{7n-5})}{(1-q^{7n-1})(1-q^{7n-6})}} = e\left(\frac{1}{14}\right) \frac{\theta {3/7 \atopwithdelims [] 1}(7\tau)}{\theta {5/7 \atopwithdelims [] 1}(7\tau)}.
\end{align}
From the infinite products in (\ref{eqn:17}) and (\ref{eqn:18}) we get 
\begin{align*}
s(\tau) t^2(\tau) = & \ q^{-1} \prod_{n \ge 1}{\frac{(1-q^{7n-3})(1-q^{7n-4})(1-q^{7n-2})^2(1-q^{7n-5})^2}{(1-q^{7n-1})^3(1-q^{7n-6})^3}}\\
= & \ h(\tau).
\end{align*}
Hence, the theta function representations in (\ref{eqn:17}) and (\ref{eqn:18}) yield
$$h(\tau) = e\left(\frac{2}{7}\right) \frac{\theta {1/7 \atopwithdelims [] 1}(7\tau) \theta^2 {3/7 \atopwithdelims [] 1}(7\tau)}{\theta^3 {5/7 \atopwithdelims [] 1}(7\tau)}.$$
To prove the formula for $h\left(\frac{2\tau-1}{7\tau-3}\right) = h(A(\tau))$ we compute the transforms
$$\theta {k/7 \atopwithdelims [] 1}(7A(\tau)), \ \textrm{for} \ k = 1, 3, 5.$$
First, we have from \cite[p. 143, (4.5)]{du}, using the mapping $B(\tau) = \frac{2\tau-7}{\tau-3}$, that
\begin{align*}
\theta {1/7 \atopwithdelims [] 1}(7A(\tau)) & = \theta {1/7 \atopwithdelims [] 1}\left(\frac{2(7\tau)-7}{(7\tau)-3}\right)\\
& = \kappa_1 \sqrt{7\tau-3} \cdot \theta {-5/7 \atopwithdelims [] 17} (7\tau);
\end{align*}
where
$$\kappa_1 = e\left(\frac{-27}{4}+\frac{37}{56}\right) \kappa_0 = e\left(\frac{-5}{56}\right) \kappa_0,$$
and $\kappa_0$ is a fixed $8$-th root of unity depending only on the mapping $B$.
Now use \cite[(4.3)]{du} with $\ell = 0$ and $m=9$ and the lower sign, according to which
$$\theta {-5/7 \atopwithdelims [] 17} (7\tau)=e\left(\frac{-5 \cdot 9}{14}\right) \theta {5/7 \atopwithdelims [] 1} (7\tau) = e\left(\frac{-3}{14}\right)\theta {5/7 \atopwithdelims [] 1} (7\tau).$$
This gives that
$$\theta {1/7 \atopwithdelims [] 1}(7A(\tau)) =\kappa_0 e\left(\frac{-17}{56}\right) \sqrt{7\tau-3} \cdot \theta {5/7 \atopwithdelims [] 1} (7\tau).$$
Next, we have 
\begin{align*}
\theta {3/7 \atopwithdelims [] 1}(7A(\tau)) & = \theta {3/7 \atopwithdelims [] 1}\left(\frac{2(7\tau)-7}{(7\tau)-3}\right)\\
& = \kappa_3 \sqrt{7\tau-3} \cdot \theta {-1/7 \atopwithdelims [] 15} (7\tau);
\end{align*}
where
$$\kappa_3 = e\left(\frac{-39}{4}+\frac{81}{56}\right) \kappa_0 = e\left(\frac{39}{56}\right) \kappa_0.$$
Using \cite[(4.3)]{du} with $\ell = 0$ and $m=8$ and the lower sign gives 
$$\theta {-1/7 \atopwithdelims [] 15} (7\tau)=e\left(\frac{-4}{7}\right) \theta {1/7 \atopwithdelims [] 1} (7\tau);$$
from which we obtain
$$\theta {3/7 \atopwithdelims [] 1}(7A(\tau)) =\kappa_0 e\left(\frac{7}{56}\right) \sqrt{7\tau-3} \cdot \theta {1/7 \atopwithdelims [] 1} (7\tau).$$
Finally, for the third theta function we have
\begin{align*}
\theta {5/7 \atopwithdelims [] 1}(7A(\tau)) & = \theta {5/7 \atopwithdelims [] 1}\left(\frac{2(7\tau)-7}{(7\tau)-3}\right)\\
& = \kappa_5 \sqrt{7\tau-3} \cdot \theta {3/7 \atopwithdelims [] 13} (7\tau);
\end{align*}
where
$$\kappa_5 = e\left(\frac{-51}{4}+\frac{141}{56}\right) \kappa_0 = e\left(\frac{-13}{56}\right) \kappa_0.$$
Now using \cite[(4.3)]{du} with $\ell=0, m=-6$ and the upper sign gives
$$\theta {3/7 \atopwithdelims [] 13} (7\tau)=e\left(\frac{2}{7}\right) \theta {3/7 \atopwithdelims [] 1} (7\tau);$$
from which we obtain
$$\theta {5/7 \atopwithdelims [] 1}(7A(\tau)) =\kappa_0 e\left(\frac{3}{56}\right) \sqrt{7\tau-3} \cdot \theta {3/7 \atopwithdelims [] 1} (7\tau).$$
Putting this all together gives
\begin{align*}
h(A(\tau)) &= e\left(\frac{2}{7}\right) \frac{\kappa_0 e\left(\frac{-17}{56}\right) \sqrt{7\tau-3} \cdot \theta {5/7 \atopwithdelims [] 1} (7\tau) \cdot \kappa_0^2 e\left(\frac{14}{56}\right) (7\tau-3) \cdot \theta^2 {1/7 \atopwithdelims [] 1} (7\tau)}{\kappa_0^3 e\left(\frac{9}{56}\right) (7\tau-3)^{3/2} \cdot \theta^3 {3/7 \atopwithdelims [] 1} (7\tau)}\\
& = e\left(\frac{1}{14}\right) \frac{\theta {5/7 \atopwithdelims [] 1} (7\tau)\cdot \theta^2 {1/7 \atopwithdelims [] 1} (7\tau)}{\theta^3 {3/7 \atopwithdelims [] 1} (7\tau)}.
\end{align*}
Now using the product formula \cite[(4.8)]{du} for the theta functions $\theta {k/7 \atopwithdelims [] 1}(7\tau)$ yields that
\begin{equation}
h\left(\frac{2\tau-1}{7\tau-3}\right) = \prod_{n \ge 1}{\frac{(1-q^{7n-1})(1-q^{7n-6})(1-q^{7n-3})^2(1-q^{7n-4})^2}{(1-q^{7n-2})^3(1-q^{7n-5})^3}}.
\label{eqn:19}
\end{equation}

On the other hand, from the relation $s^7(\tau) = h(\tau) (h(\tau)-1)^2$ we easily derive the product formula
$$h(\tau)-1 = q^{-1} \prod_{n \ge 1}{\frac{(1-q^{7n-3})^3(1-q^{7n-4})^3}{(1-q^{7n-1})^2(1-q^{7n-6})^2(1-q^{7n-2})(1-q^{7n-5})}}.$$
Dividing by the product formula for $h(\tau)$ yields that
$$\frac{h(\tau)-1}{h(\tau)} = \prod_{n \ge 1}{\frac{(1-q^{7n-1})(1-q^{7n-6})(1-q^{7n-3})^2(1-q^{7n-4})^2}{(1-q^{7n-2})^3(1-q^{7n-5})^3}}$$
and proves the formula.  From (\ref{eqn:16}) it follows that
$$h\left(\frac{3\tau-1}{7\tau-2}\right) = h(A^2(\tau)) = \frac{h(A(\tau))-1}{h(A(\tau))} = \frac{1}{1-h(\tau)},$$
which is formula (\ref{eqn:15}).  This proves Theorem \ref{thm:11}.
\end{proof}

\noindent {\bf Remark.} Alternatively, we could finish the above proof by noting that $h\left(\frac{2\tau-1}{7\tau-3}\right)$ is a modular function for $\Gamma_1(7)$, since $T \in \Gamma_1(7)$ satisfies
$$h(A T(\tau)) = h(SA(\tau)) = h(A(\tau)), \ \ \textrm{for some} \  S \in \Gamma_1(7).$$
Since $h(\tau)$ is a Hauptmodul for $\Gamma_1(7)$ and $h(\tau) \rightarrow h(A(\tau))$ induces an automorphism of $\textsf{K}_{\Gamma_1(7)}/\textsf{K}_{\Gamma_0(7)}$, it follows that $h(A(\tau))$ is also a Hauptmodul and therefore equal to a linear fractional expression in $h(\tau)$.  Furthermore, the six values of $h(\tau)$ at the cusps of $\Gamma_1(7)$ (namely, $\infty, 1 ,0$ and the roots $r_i$ of $x^3-8x^2+5x+1$, by (\ref{eqn:5})), are permuted by this automorphism (as residues of $h$ modulo the prime divisors of $\textsf{K}_{\Gamma_1(7)}$ at infinity).  Since the value $\infty$ corresponding to $\tau = \infty i$ is mapped to $h(A(\tau)) = 1$, by the product formula (\ref{eqn:19}), it follows that the value $1$ must be mapped to $0$.  This holds because $h(A(\tau)) = \frac{h+a}{h+b}$ cannot map a root $r_i$ to $0$: otherwise $a = -r_i$, and applying the map twice sends $1$ to $\infty$ ($A$ has order $3$), which would yield that $b$ satisfies the irreducible equation $b^2+b+1-r_i = 0$ over $\mathbb{Q}(\zeta_7)$ (the norm of its discriminant is $-43$).  But then $\frac{1-r_i}{1+b} = r_j$, for some $j$, would be impossible and $\frac{h+a}{h+b}$ could not map $1$ to one of the $r_j$.  Hence, $a = -1, b = 0$ and $h(A(\tau)) = \frac{h-1}{h}$.  From the resulting product formula for $(h(\tau)-1)/h(\tau)$ we can derive the equation $s^7(\tau) = h(\tau) (h(\tau)-1)^2$. \medskip

Now we use the fact that $[\Gamma_0(7): \Gamma_1(7) \cup (-I)\Gamma_1(7)]=3$, from which it follows that $1, A, A^2$ are representatives for the cosets of $\Gamma_1[7] = \Gamma_1(7) \cup (-I)\Gamma_1(7)$ in $\Gamma_0(7)$.  It follows that the function
\begin{align}
\notag z(\tau) &= h(\tau)+h(A(\tau))+h(A^2(\tau))\\
& = h(\tau)+\frac{h(\tau)-1}{h(\tau)}+\frac{1}{1-h(\tau)} = \frac{h^3(\tau)-3h(\tau)+1}{h(\tau)(h(\tau)-1)}
\label{eqn:20}
\end{align}
is a modular function for $\Gamma_0(7)$ with a simple pole at $\infty i$ and the value $z(\tau) = z(0) = 8$ at the other cusp of $\Gamma_0(7)$ (since the values of $h(\tau)$ at the cusps of $\Gamma_1(7)$ lying above $0$ are the roots $r_i$ of $x^3-8x^2+5x+1$).  It is clear that $\left(\frac{\eta(\tau)}{\eta(7\tau)}\right)^4$ is a Hauptmodul for $\Gamma_0(7)$  \cite[pp. 46, 51]{sch}, and comparing $q$-expansions gives that
$$z(\tau) = \left(\frac{\eta(\tau)}{\eta(7\tau)}\right)^4+8.$$
This shows that
\begin{align*}
\left(\frac{\eta(\tau)}{\eta(7\tau)}\right)^4 &= z(\tau)-8 =  \frac{h^3(\tau)-3h(\tau)+1}{h(\tau)(h(\tau)-1)}-8\\ 
& = \frac{h^3(\tau)-8h^2(\tau)+5h(\tau)+1}{h(\tau)(h(\tau)-1)},
\end{align*}
which is (\ref{eqn:13}).  See \cite[(4.24), p. 89]{elk}. \medskip

These identities are closely related to several of Ramanujan's entries in the unorganized material of his Notebooks.  See entries 31 and 32 in \cite[pp. 174-184]{ber}.  In particular, the product representation of $h(\tau)-1$ in the above proof is equivalent to Entry 32(ii) of \cite[p. 176]{ber} (or \cite[(1.2)]{bz}); and the relation $f_7(h(\tau),j_7^*(\tau)) = 0$ in Section 4 is, assuming Theorem 10, equivalent to Entry 32(iii) of \cite[p. 176]{ber} (or \cite[(1.3)]{bz}).  Also see \cite[Thm. 7.14, p. 440]{cp} and \cite{mor5}.

\noindent Dept. of Mathematical Sciences, LD 270

\noindent Indiana University -- Purdue University at Indianapolis (IUPUI)

\noindent 402 N. Blackford St., Indianapolis, IN 46202

\noindent e-mail: pmorton@iupui.edu

\end{document}